\documentclass{amsart}
\usepackage{r}  

\begin{document}
\title{Boundary representations from constrained interpolation}

\author[G. Ben Ayun]{Gal Ben Ayun}
\address{Department of Mathematics\\ Ben-Gurion University of the Negev\\
Beer-Sheva \; 8410501 \\ Israel}

\author[E. Shamovich]{Eli Shamovich}
\address{Department of Mathematics\\ Ben-Gurion University of the Negev\\
Beer-Sheva \; 8410501 \\ Israel} 

\thanks{E.S. was partially supported by BSF Grant no. 2022235}

\begin{abstract}
   In this paper, we study $C^*$-envelopes of finite-dimensional operator algebras arising from constrained interpolation problems on the unit disc. In particular, we consider interpolation problems for the algebra \( H^\infty_{\text{node}}\) that consists of bounded analytic functions on the unit disk that satisfy \( f(0) = f(\lambda) \) for some \( 0 \neq \lambda \in \D \). We show that there exist choices of four interpolation nodes that exclude both $0$ and $\lambda$, such that if $I$ is the ideal of functions that vanish at the interpolation nodes, then $C^*_e(H^\infty_{\text{node}}/I)$ is infinite-dimensional. This differs markedly from the behavior of the algebra corresponding to interpolation nodes that contain the constrained points studied in the literature. Additionally, we use the distance formula to provide a completely isometric embedding of $C^*_e(H^\infty_{\text{node}}/I)$ for any choice of $n$ interpolation nodes that do not contain the constrained points into $M_n(G^2_{nc})$, where $G^2_{nc}$ is Brown's noncommutative Grassmannian.
\end{abstract}
\maketitle

\section{Introduction}
     Let $\{z_1,\ldots,z_n\}\subset\bb{D} = \left\{z \in \C \mid |z| < 1\right\}$ be a finite set of points and $\{w_1,\ldots,w_n\}\subset\bb{C}$. The Pick interpolation problem is the following: does there exist a holomorphic function $f \colon \bb{D}\to\overline{\bb{D}}$ (called an \textit{interpolating function}) such that $f(z_i)=w_i$? The solution to this problem is well known - a necessary and sufficient condition for the existence of an interpolating function is that the \textit{Pick matrix}
    $$\Bigg[\frac{1-w_i\overline{w_j}}{1-z_i\overline{z_j}}\Bigg]_{i,j=1}^n$$ is positive semi-definite. This question was solved independently by Pick and Nevanlinna (See, for example, \cite{pick1915beschrankungen}). More generally, one can formulate the matricial Pick interpolation problem, where the $w_i$'s are replaced by $s\times r$ matrices $\{W_1,\ldots, W_n\}$. Analogously, an interpolating function $f$ with $f(z_i)=W_i$ and $\|f\|\leq 1$ exists if and only if
    $$\Bigg[\frac{I-W_iW_j^*}{1-z_i\overline{z_j}}\Bigg] \geq 0.$$
   The operator algebraic approach to the Pick interpolation problem was pioneered by Koranyi and Sz-Nagy \cite{sz1958operatortheoretische}, and a breakthrough was achieved by Sarason \cite{sarason1967generalized}. The key ingredient is the relationship between $H^2(\D)$ and $H^{\infty}(\D)$. Recall that $H^2(\D)$ is the reproducing kernel Hilbert space (RKHS for short) of all analytic functions on $\D$ with square summable MacLaurin coefficients. The multiplier algebra of $H^2(\D)$ is $H^{\infty}(\D)$, the algebra of all bounded analytic function on $\D$. One encodes the fact that one can test the existence of solutions by the positivity of the corresponding Pick matrix for all sizes of matrices by saying that $H^2(\D)$ is a complete Pick space. Complete Pick spaces are well-studied and important objects in function theory. A celebrated theorem of Agler and McCarthy \cite{AglerMcC-compPick} tells us that under some mild assumptions, every complete Pick space is equivalent as an RKHS to a quotient of the Drury-Arveson space. Namely, the RKHS of analytic functions on the unit ball of $\C^d$ (where we allow $d = \infty$) with the reproducing kernel $k(z,w) = \frac{1}{1 - \langle z, w \rangle}$.
    
   However, even in the case of the annulus, the Pick interpolation problem does not have such a simple and elegant solution. Abrahamse \cite{abrahamse1979pick} has discovered that to solve the scalar Pick interpolation problem on a multiply-connected domain of genus $g$, one requires a family of kernels parametrized by the $g$-dimensional torus. In \cite{Ball}, Ball has solved the matricial interpolation problem on multiply-connected domains. However, a sufficient family of kernels described by Ball is parametrized by $g$-tuples of unitaries of all sizes. In \cite{Agler-unpub}, Agler formulated a general approach to the Pick interpolation problem based on families of kernels (see also \cite{JKMcC}). 
    
    Another generalization is the constrained Pick interpolation problem, where the interpolating function is to satisfy an additional condition. In \cite{davidson2009constrained}, the authors solve an interpolation problem on $\D$ with the additional constraint that the interpolating function $f$ satisfies $f'(0)=0$. The algebra of all such functions is denoted $H^{\infty}_1$. In this case, a family of kernel functions
    $$\{k^{\alpha,\beta}\}_{|\alpha|^2+|\beta|^2=1}$$ is required. Ragupathi \cite{ragupathi2009nevanlinna} generalized this result for subalgebras of $H^{\infty}$ with \textit{predual factorization}. One main family of examples of subalgebras with predual factorization consists of subalgebras of $H^{\infty}(\D)$ of the form
    $$H_{B}=\bb{C} \cdot 1+BH^{\infty},$$
    where $B$ is a Blaschke product. The constrained interpolation problem is analogous to the interpolation problem in multiply-connected domains. To be more precise, if one considers multiply-connected domains, the corresponding algebro-geometric object is a curve of genus $g$. In the constrained case (with a finite Blaschke product), the algebro-geometric object is a singular rational curve. One of the first examples of the interplay between dilation theory and singular rational curves is the Neal parabola \cite{DJMcC}. The Neal parabola is the planar curve cut out by the equation $y^2 = x^3$. The map $z \mapsto (z^2, z^3)$ gives a parametrization of the Neal parabola. It is rather straightforward that the algebra $H^{\infty}_1$ corresponds to the Neal parabola. More generally, $H_{B}$ corresponds to a parametrization of a rational curved ramified at the zeroes of $B$ with multiplicities prescribed by $B$. In \cite{ball2010constrained}, Ball, Bolotnikov, and ter Horst solved the matrix-valued constrained interpolation problem for $H^{\infty}_1$. An approach to the constrained interpolation problem through test functions was taken in \cite{DritschelPickering, DritschelUndrakh}. In \cite{davidson2011nevanlinna}, an approach was taken using the $\A_1(1)$ property of dual operator algebras, which was subsequently developed in \cite{davidson2020nevanlinna} with the introduction of the conductor ideal.

    The minimality of the corresponding family of kernels can be studied using boundary representations. Arveson introduced boundary representations in \cite{arveson1969subalgebras}. Given an operator algebra $A$, one would like to find the smallest $C^*$-algebra generated by $A$. This $C^*$-algebra is called the $C^*$-envelope of $A$. Its existence was conjectured by Arveson and proved by Hamana \cite{hamana1979injective}. In \cite{mccullough2002c}, Paulsen and McCullough calculate the $C^*$-envelopes of two operator algebras - the first one is
    $H^{\infty}(\bb{D}^2)/{I}$ with
    $$I=\left\{f\in H^{\infty}(\bb{D}^2) \mid f(0,0)=f\left(\frac{1}{\sqrt{2}},0\right)=f\left(0,\frac{1}{\sqrt{2}}\right)=0\right\}.$$
    This operator algebra arises from an interpolation problem studied by Solazzo \cite{solazzo2000interpolation}. The second algebra is $H^{\infty}(\bb{A})/{I_F}$, which arises from an interpolation problem in finitely connected domains, studied by Abrahamse \cite{abrahamse1979pick}. Here, $F \subset \bb{A}$ is the finite set of interpolation nodes and $I_F$ is the ideal of functions in $H^{\infty}(\bb{A})$ that vanish on $F$. It turns out that $C^*_e(H^{\infty}(\bb{A})/I_F) \cong M_n(C(\T))$, where $3 \leq n = |F|$ and $\T = \partial \D$ is the unit circle. Since, in the case of the annulus, the kernel family is parametrized by $\T$, we see that this family is minimal and solves the complete Pick interpolation problem on the annulus. In \cite{davidson2009constrained}, the authors show that if $0 \in F$, then $C^*_e(H^{\infty}_1/I_F) \cong M_{n+1}$. More generally, Ragupathi \cite{ragupathi2009nevanlinna} calculated $C^*_e(H_B/I_F)$ in the special case that $F$ contains some of the zeroes of $B$. The $C^*$-envelopes always turned out to be algebras of matrices. 

In this paper, we study constrained interpolation algebras arising from nodal cubics. Namely, for $0 \neq \lambda \in \D$, we consider the algebra
    $$\node=\{f\in H^{\infty} \mid f(0)=f(\lambda)\}=\bb{C}\cdot 1+\frac{z(z-\lambda)}{1-\overline{\lambda}z}H^{\infty}.$$
    We describe the solution to the matricial Pick interpolation problem and the corresponding family of kernels in Section \ref{sec:pick}. The family is analogous to the one obtained by Ball, Bolotnikov, and ter Horst \cite{ball2010constrained}. The first main result of the paper is in  Section \ref{sec:boundary}. 
    \begin{thmA}[{Theorem \ref{thm:good_then_big_envelope}} and {Theorem \ref{thm:example_good}}]
    For $\lambda = \frac{1}{\sqrt{2}}$, there exists $F = \{z_1,z_2,z_3,z_4\} \subset \D$, such that $C^*_e(\node/I_F)$ is infinite-dimensional.
    \end{thmA}
    To this end, we produce a family of pairwise unitarily inequivalent dilation maximal irreducible representations parametrized by the circle without a finite set of points. 
    This result shows that the behavior of the $C^*$-envelope when the set of interpolation nodes does not contain the ramification points is markedly different from the case studied in the literature so far. In Section \ref{sec:cover}, we provide a candidate for the $C^*$-envelope of $C^*_e(\node/I_F)$ for $F$ that does not contain $0$ and $\lambda$. 
    \begin{thmA}[{Theorem \ref{thm:completely.isometric.embedding}}]
        Let $0, \lambda \notin F \subset \D$ be a finite set. Then, there exists a completely isometric embedding $\Gamma \colon \node/I_F \to M_n(G_2^{nc})$.
    \end{thmA}
    Here $G^2_{nc}$ is the noncommutative Grassmannian introduced by Brown \cite{brown1981ext}. The noncommutative Grassmannian $G^2_{nc}$ is the noncommutative analog of the Riemann sphere. As we will see, the latter parametrizes the family of kernels for the scalar Pick problem.
    
    \vspace{0.2cm}
    
    \textbf{Acknowledgment:} We thank the referees for reading the manuscript carefully and providing numerous comments. These comments have contributed greatly to the presentation of this paper.

\section{Pick family for $H^{\infty}_{\operatorname{node}}$}\label{sec:pick}

The content of this section is inspired by the work of Ball, Bolotnikov, and ter Horst \cite{ball2010constrained}. Though some proofs are similar to \cite{ball2010constrained}, we include them for completeness. Our method differs from the approach in \cite{ball2010constrained}, and we aim to obtain a distance formula.

 For $k,\,l\in\N$, we denote by $M_{k\times l}(L^2)$ the Hilbert space of $k\times l$ matrices with entries in $L^2(\T)$, with the inner product given by
    $$\ip{F}{G}=\frac{1}{2\pi}\int_{\T}\operatorname{tr}[FG^*].$$
    Similarly, $M_{k\times l}(L^1)$ is the Banach space of $k\times l$ matrices with entries in $L^1$, with the norm given by
    $$\|F\|_1=\frac{1}{2\pi}\int_{\T}\operatorname{tr}[(F^*F)^{\frac{1}{2}}],$$
    and $M_{k\times l}(L^{\infty})$ is the Banach space of $k\times l$ matrices with entries in $L^{\infty}
    $ and the sup norm $\|\cdot\|_{\infty}$.
    We will also need the Hilbert space $M_{k\times l}(H^2)$ and the Banach spaces $M_{k\times l}(H^1)$ and $M_{k\times l}(H^{\infty})$, which are the subspaces of analytic functions in $M_{k\times l}(L^2)$, $M_{k\times l}(L^1)$ and $M_{k\times l}(L^{\infty})$, respectively.
    We recall that the Riesz representation theorem identifies $M_{k\times l}(L^1)^*$ with $M_{l\times k}(L^{\infty})$ via
    $$\varphi_F(G)=\frac{1}{2\pi}\int_{\T}\operatorname{tr}[FG],$$ for $F\in M_{k\times l}(L^1)$ and $G\in M_{l\times k}(L^{\infty})$. Moreover, the preannihilator of $M_{k\times l}(H^{\infty})$ is given by $M_{l\times k}(zH^1)$.
    \vspace{0.25cm}

For a fixed $0\neq \lambda\in\D$, consider the norm-closed, co-dimension $1$ subalgebra of $H^{\infty}$ defined by
    $$\node=\{f\in H^{\infty} \mid f(0)=f(\lambda)\}=\C\cdot 1+B_{\lambda}H^{\infty},$$ where $B_{\lambda}(z)=z\frac{z-\lambda}{1-\overline{\lambda}z}$ is the Blaschke product vanishing at $0$ and $\lambda$.
    Moreover, for $z\in\D$ let $k_w(z)=\frac{1}{1-z\overline{w}}$ be the function in $H^2$ satisfying $\ip{f}{k_w}=f(w)$ for all $f\in H^2$. 
    \begin{lem}
        $\node$ separates points in $\bb{D}\setminus\{0,\lambda\}$.
    \end{lem}
    \begin{proof}
        Let $z_1,\ldots,z_n\in\bb{D}\setminus\{0,\lambda\}$. By Lagrange interpolation, there are polynomials $\{f_i\}_{i=1}^n\subset H^{\infty}$ satisfying
        $$f_i(z_j)=\delta_{i,j},\,f_i(0)=f_i(\lambda)=0.$$
        In particular, $f_i\in \node$ for all $1\leq i\leq n$.
    \end{proof}
    
    \begin{lem}
        Let $(\alpha,\beta)\in\C^2$ such that $|\alpha|^2+|\beta|^2=1$ and define
        $$H^2_{\alpha,\beta}=\{f\in H^2 \mid \exists c \in \C,\, f(0)=c\cdot\alpha,\,f(\lambda)=c\cdot(\alpha+f_{\lambda}(\lambda)\beta)\}.$$
        Here $f_{\lambda}(z)=\frac{\sqrt{1-|\lambda|^2}}{|\lambda|}(k_{\lambda}(z)-1)$. Then, $H^2_{\alpha,\beta}$ is a reproducing kernel Hilbert space with kernel function
        $$k^{\alpha,\beta}(z,w)=\overline{(\alpha+\beta f_{\lambda}(w))}(\alpha+\beta f_{\lambda}(z))+\frac{B_{\lambda}(z)\overline{B_{\lambda}(w)}}{1-z\overline{w}}.$$ \label{prop.reprod.kernel}
    \end{lem}
    \begin{proof}

        Note that $f_{\lambda}(\lambda) = \frac{|\lambda|}{\sqrt{1 - |\lambda|^2}}$. Hence, for every $g \in H^2_{\alpha,\beta}$,
        \[
        \langle g, 1 \rangle = c \alpha,\, \langle g, f_{\lambda} \rangle = \frac{\sqrt{1- |\lambda|^2}}{|\lambda|} (g(\lambda) - g(0)) = c \beta.
        \]
        In other words, $H^2_{\alpha,\beta}=\operatorname{span}\{\alpha+\beta f_{\lambda}(z)\}\oplus B_{\lambda}H^2$, where the sum is direct since $B_{\lambda}H^2$ is orthogonal to $\operatorname{span}\{1,f_{\lambda}\}=\operatorname{span}\{1,k_{\lambda}\}$, as $B_{\lambda}(0)=B_{\lambda}(\lambda)=0$.
        Therefore, $H^2_{\alpha,\beta}\subset H^2$ is a closed subspace. Hence, $H^2_{\alpha,\beta}$ is a RKHS with kernel functions $P_{H^2_{\alpha, \beta}} k_w$. Moreover,
        \begin{multline*}
        k^{\alpha,\beta}(z,w) = \langle P_{H^2_{\alpha, \beta}} k_w, P_{H^2_{\alpha, \beta}} k_z \rangle  = \langle k_w, \alpha + \beta f_{\lambda} \rangle (\alpha + \beta f_{\lambda}(z)) + \langle M_{B_{\lambda}} M_{B_{\lambda}}^* k_w, k_z \rangle = \\ \overline{(\alpha + \beta f_{\lambda}(w))}(\alpha + \beta f_{\lambda}(z)) + \frac{B_{\lambda}(z)\overline{B_{\lambda}(w)}}{1 - z \bar{w}}.
        \end{multline*}

    \end{proof}

    \begin{lem} For every $|\alpha|^2+|\beta|^2=1$, $\node \subseteq \operatorname{Mult}(H^2_{\alpha,\beta})$. Moreover, if $\alpha\neq 0, -f_{\lambda}(\lambda)\beta$, then we have equality. As a result, for every $|\alpha|^2+|\beta|^2=1$ and every $f \in \node$,  $M_f^*k^{\alpha,\beta}_{z_i}=\overline{f(z_i)}k^{\alpha,\beta}_{z_i}.$\label{lem.eigenvalues.of.multiplier}
    \end{lem}
    \begin{proof}
        Fix arbitrary $(\alpha, \beta)$ on the unit sphere. Let $f\in \node$ and $f(0) = f(\lambda) = \tilde{c}$. Then, for every $g\in H^2_{\alpha,\beta}$,
        $$(fg)(0)=c f(0)\alpha = c \tilde{c} \alpha,\,(fg)(\lambda)=c \tilde{c} (\alpha+f_{\lambda}(\lambda)\beta).$$ That is $fg\in H^2_{\alpha,\beta}$, hence $f\in \operatorname{Mult}(H^2_{\alpha,\beta})$. The last part then follows automatically by properties of $\operatorname{Mult}(\mathcal{H})$ for a general RKHS.
        \\\\
        On the other hand, for every $z_0 \in \D \setminus \{0, \lambda\}$,
        \[
        k^{\alpha,\beta}_{z_0}(z_0) = \|k^{\alpha,\beta}_{z_0}\|^2 = |\alpha + f_{\lambda}(z_0) \beta|^2 + \frac{|B_{\lambda}(z_0)|^2}{1 - |z_0|^2} > 0.
        \]
        Let $f\in \operatorname{Mult}(H^2_{\alpha,\beta})$. Since $f$ is a multiplier, $\underset{z\in\D}{\sup}|f(z)|\leq \|M_f\|_{\text{mult}}<\infty$.
        Let $z_0\in\D \setminus \{0, \lambda\}$. Since $k^{\alpha,\beta}_{z_0}(z_0)> 0$ and is analytic, $\frac{1}{k^{\alpha,\beta}_{z_0}(z)}$ is analytic in a neighborhood of $z_0$. Moreover, $fk_{z_0}^{\alpha,\beta}\in H^2_{\alpha,\beta}$, and thus, $f=\frac{fk_{z_0}^{\alpha,\beta}}{k_{z_0}^{\alpha,\beta}}$ is analytic in a neighborhood of $z_0$. By Riemann's removable singularity theorem, $f$ is analytic in all of $\D$. Therefore, $f\in H^{\infty}$. Assume, in addition, that $\alpha \neq 0, - f_{\lambda}(\lambda) \beta$. Then, let $p\in H^2_{\alpha,\beta}$ be a polynomial satisfying
        $$p(0)=\alpha,\,p(\lambda)=\alpha+f_{\lambda}(\lambda)\beta.$$
        By assumption,
        $$f(0)p(0)=c\cdot\alpha,\,f(\lambda)p(\lambda)=c\cdot (\alpha+f_{\lambda}(\lambda)\beta).$$ Therefore, $f(0)=f(\lambda)=c$, that is $f\in \node$.
        
    \end{proof}
     
    For fixed $k,\,l'\geq l\in\N$, let $M_k(\node)$ denote the space of $k\times k$ matrices with entries in $\node$ and define
    $$\G(l'\times l)=\{\alpha,\beta\in M_{l'\times l}(\C) \mid \alpha\alpha^*+\beta\beta^*=I_l\}.$$
    Furthermore, for $(\alpha,\beta)\in\G(l'\times l)$, define
    $$(H^2_{\alpha,\beta})_{k,l,l'}=\{F\in M_{k\times l}(H^2) \mid \exists M\in M_{k\times l'}(\C)  \mid F(0)=M\alpha,\,F(\lambda)=M(\alpha+f_{\lambda}(\lambda)\beta)\}.$$
    To simplify notations, we will shorten this by $(H^2_{\alpha,\beta})_k$, if $k=l=l'$ and if $k=l=l'=1$, we will simply write $H^2_{\alpha,\beta}$.
    By an argument similar to the one in the proof of Lemma \ref{prop.reprod.kernel},
    \[(H^2_{\alpha,\beta})_{k,l,l'}=M_{k\times l'}(\C)(\alpha+\beta f_{\lambda}(z))\oplus B_{\lambda}M_{k\times l}(H^2).\]

    \begin{prop}
        $(H^2_{\alpha,\beta})_{k,l,l'}$ is a matrix-valued RKHS with kernel function
        $$k^{\alpha,\beta}(z,w)=(\alpha^*+\overline{f_{\lambda}(w)}\beta^*)(\alpha+f_{\lambda}(z)\beta)+\frac{B_{\lambda}(z)\overline{B_{\lambda}(w)}}{1-z\overline{w}}I_l.$$
        To be more precise, for every $f \in (H^2_{\alpha,\beta})_{k,l,l'}$, every $w \in \D$, and every $X \in M_{k\times l}(\C)$,
        \[
        \tr(f(w)X^*) = \langle f, X k^{\alpha,\beta}_w \rangle.
        \]
        Moreover, for every $F\in M_k(\node)$,
        $$M_F^*Xk^{\alpha,\beta}(\cdot,w)=F(w)^*Xk^{\alpha,\beta}(\cdot,w).$$
        \label{prop.kernel.function}
    \end{prop}
    \begin{proof}
        Let $f\in (H^2_{\alpha,\beta})_{k,l,l'}$ and write $f=U(\alpha+\beta f_{\lambda})+B_{\lambda}\tilde{f}$ for some $\tilde{f}\in M_{k\times l}(H^2)$. Then, for every
        $X\in M_{k\times l}(\C)$, 
        \begin{align*}
    \ip{f}{Xk^{\alpha,\beta}(\cdot,w)} 
    &= \ip{U(\alpha+\beta f_\lambda)+B_{\lambda}\tilde{f}}{Xk^{\alpha,\beta}(\cdot,w)}= \\
    &= \underbrace{\ip{U(\alpha+\beta f_\lambda)}{X(\alpha^*+\beta^*\overline{f_{\lambda}(w)})(\alpha+\beta f_{\lambda})}}_{(1)} 
    + \underbrace{\ip{B_{\lambda}\tilde{f}}{\overline{B_{\lambda}(w)}X {B_{\lambda} k(\cdot,w)}}}_{(2)}.
    \end{align*}
    Here, we write $k(z,w) = \frac{1}{1 - z \bar{w}}$. We treat each summand separately.   
        \begin{align*}
    &\ip{U(\alpha+\beta f_\lambda)}{X(\alpha^*+\beta^*\overline{f_{\lambda}(w)})(\alpha+\beta f_{\lambda})}= \\
    &= \ip{U\alpha}{X(\alpha^*+\beta^*\overline{f_{\lambda}}(w))\alpha} 
    + \ip{U\beta f_{\lambda}}{X(\alpha^*+\beta^*\overline{f_{\lambda}(w)})\beta f_{\lambda}}= \\
    &= \operatorname{tr}\left[U\alpha\alpha^*(\alpha+\beta f_{\lambda}(w))X^*\right] 
    + \operatorname{tr}\left[U\beta\beta^*(\alpha+\beta f_{\lambda}(w))X^*\right] \\
    &= \operatorname{tr}\left[U(\alpha+\beta f_{\lambda}(w))X^*\right].
\end{align*}
For the second summand, we observe that the operator of multiplication by $B_{\lambda}$ is an isometry. Therefore, 
\[
\ip{B_{\lambda}\tilde{f}}{X\frac{B_{\lambda}\overline{B_{\lambda}(w)}}{1-z\overline{w}}} = B_{\lambda}(w) \left\langle \tilde{f}, X \frac{1}{1 - z \overline{w}}\right\rangle = B_{\lambda}(w) \operatorname{tr}(\tilde{f}(w) X^*).
\]
Combining the two, we get our result. For the last claim of the lemma, note that for every $G\in (H^2_{\alpha,\beta})_{k,l,l'}$,
$$\ip{G}{M_F^*Xk^{\alpha,\beta}(\cdot,w)}=\ip{FG}{Xk^{\alpha,\beta}(\cdot,w)}=\operatorname{tr}(X^*F(w)G(w))=\operatorname{tr}((F(w)^*X)^*G(w))=$$$$=\ip{G(w)}{F(w)^*Xk^{\alpha,\beta}(\cdot,w)}=\ip{G}{F(w)^*Xk^{\alpha,\beta}(\cdot,w)}.$$
\end{proof}
Now we are ready to present the main theorem of this section.
\begin{thm}
        For a fixed $k\in\N$, let $\{z_1,\ldots,z_n\}\subset\D$ and $\{W_1,\ldots,W_n\}\in M_k(\C)$. There exists a function $F\in M_k(\node)$ satisfying $\|F\|_{\infty} \leq 1$ and $F(z_i)=W_i$, for all $1 \leq i \leq n$ 
        if and only if for every $(\alpha,\beta)\in\G(k\times k)$ and for every  
        $(X_1,\ldots,X_n)\in (M_k(\C))^n$,
        
        \begin{equation}\sum_{i,j=1}^n \operatorname{tr}[X_jk^{\alpha,\beta}(z_i,z_j)X_i^*-W_j^*X_jk^{\alpha,\beta}(z_i,z_j)X_i^*W_i]\geq 0.\label{eq.trace.condition}\end{equation}
        \label{thm.trace.condition}
        \end{thm}
    Before proving this theorem, we need a few general results about $(H^2_{\alpha,\beta})_{k,l,l'}$. To this end, fix some points $F=\{z_1,\ldots,z_n\}\subset\D$, let $B_F$ be the Blaschke product with simple zeroes at $F$ and define the ideal
    $$I=\{f\in \node \mid \forall 1\leq i\leq n,\,f(z_i)=0\}=\node\cap B_F H^{\infty}.$$

    \begin{lem}
        For a fixed $k\in\N$ set $$\mathcal{I}=\{G\in M_k(\node) \mid \forall 1\leq i\leq n\,\ G(z_i)=0\}=M_k(\node)\cap B_FM_k(H^\infty).$$
        Then,    $$\mathcal{I}_{\perp}=M_k(I)_{\perp}\subset\overline{\varphi_{\lambda}B_F}M_k(H^1),$$ with $\varphi_{\lambda}(z)=\frac{z-\lambda}{1-\overline{\lambda}z}=\frac{B_{\lambda}(z)}{z}$.
    \end{lem}
    \begin{proof}
        It is known that $(H^{\infty})_{\perp}=H^1_0=zH^1$. Since $B_{\lambda}B_F H^{\infty}\subset I$ and since taking pre-annihilators is inclusion reversing,
        $$I_{\perp}\subset (B_{\lambda}B_F H^{\infty})_{\perp}=\overline{B_{\lambda}B_F}zH^1=\overline{\varphi_{\lambda}B_F}H^1.$$
        Therefore,
        $$\mathcal{I}_{\perp}=M_k(I)_{\perp}=M_k(I_{\perp})\subset\overline{\varphi_{\lambda}B_F}M_k(H^1).$$
    \end{proof}
    \begin{lem}
        The set $$S=\{h\in \mathcal{I}_{\perp}
        \mid \exists s\in\D,\, \det(h(s))\neq 0\}\subset \mathcal{I}_{\perp}$$ is dense in $\mathcal{I}_{\perp}$.
    \end{lem}
    \begin{proof}
        Fix $\varepsilon>0$ and let $f\in \mathcal{I}_{\perp}$ be such that for every $s$, $\det(f(s))=0$. Note that
        $f(\frac{1}{2})$ has a finite spectrum. Hence there exists $0<\delta$ such that $\frac{\delta}{2}\not\in\sigma(f(\frac{1}{2}))$ and $\delta\cdot k<\varepsilon$. Now, $f-\delta zI_k$ is invertible at $\frac{1}{2}$ and $f-\delta zI_k\in S$ as $f,\,\delta zI_k\in\mathcal{I}_{\perp}$. Moreover,
        $$\|f-(f-\delta z I_k)\|_1=\delta\cdot \|I_k\|_1=\delta\cdot k<\varepsilon.$$
    \end{proof}

    Recall from \cite{nagy2010harmonic}, that a function $f\in M_{k \times \ell}(H^{\infty})$ is called \textit{inner} if $f(e^{i\theta})$ is an isometry almost everywhere on $\T$ (in particular, $\ell \leq k$). Additionally, we will say that $g \in M_{s \times t}(H^2)$ is \text{right outer} (respectively, \text{left outer}), if $\overline{M_s(H^{\infty})g}=M_{s\times t}(H^2)$ (respectively, $\overline{g M_t(H^{\infty})}=M_{s\times t}(H^2)$.) More generally, by \cite[Definition B in Section 5.2]{RosenblumRovnyak}, a holomorphic operator-valued function $g$ is outer, if there exists a bounded analytic (in fact, outer) function $\varphi$, such that $\varphi g$ is bounded and outer. Though the definition in \cite{RosenblumRovnyak} of an outer function slightly differs from ours, they boil down to the same thing, if one restricts attention to the subspace they call $M_{out}$. Moreover, by \cite[Theorem A in Section 5.4]{RosenblumRovnyak}, the definition is independent of $\varphi$ and by \cite[Theorem B in Section 5.4]{RosenblumRovnyak}, an outer function has pointwise dense range. Since we only consider finite-dimensional Hilbert spaces, the latter is just surjectivity. 
    
    \begin{lem}
        Let $g\in M_k(H^\infty)$, such that $g^T$ is left outer. Then, $\det(g(z))\neq 0$ for all $z\in\D$. A similar result holds for the Nevanlinna class outer functions.
    \end{lem}
    \begin{proof}
        Let $K(z,w)=\frac{1}{1-z\overline{w}}I$ be matricial Szego kernel.
        As $g^T$ is left outer, there exists a sequence $\{f_n\}\subset M_k(H^\infty)$, such that $\lim_{n\to \infty} g^Tf_n = I_k$ in $M_k(H^2)$. Hence, for every $z\in\D$ and $1\leq i,j\leq k$,
        $$(g^T(z)f_n(z))_{i,j}=\ip{g^T(z)f_n(z)}{E_{i,j}}=\ip{g^Tf_n}{E_{i,j}K(\cdot,z)}\to \ip{I_k}{E_{i,j}K(\cdot,z)}=\ip{I_k}{E_{i,j}}=(I_k)_{i,j},$$
        i.e., $g^T(z)f_n(z)\to I_k$ entry-wise. By continuity of the determinant,
        $$\det(g^T(z)f_n(z))\to\det(I_k)=1.$$
        Therefore, $\det(g^T(z))=\det(g(z))\neq 0$. The second claim follows immediately by multiplying by a non-vanishing scalar-valued bounded function that will make $g$ bounded.
    \end{proof}
     The following factorization result in $M_k(H^1)$ is, probably, well-known to experts. However, since we do not have a good reference, we include the proof here.
     \begin{lem}
        Every function $h\in M_k(H^1)$ can be factored as \[h=g\cdot f\] with $f\in M_{l\times k}(H^2)$,\,$g\in M_{k\times l}(H^2)$,\, $l \leq k$, $g^T$ left outer, and $\|g\|_2^2=\|f\|^2_2=\|h\|_1$.\label{thm.factorization.M_k(H^1)}
    \end{lem}
    \begin{proof}
        By \cite[Theorem in Section 5.7]{RosenblumRovnyak}, $h^T$ can be
        factored as
        $h^T=h_i^Th_o^T$, with $h_i^T$ a $k\times l'$ inner function and $h_o^T\in M_{l'\times k}(H^1)$ (for some $l'\leq k$) a left outer function. Note that, though the theorem of Rosenblum and Rovnyak yields a factorization in the Nevanlinna class, by \cite[Theorem A in Section 4.7]{RosenblumRovnyak}, the Hardy classes are the intersection of the Nevanlinna class with the corresponding $L^p$ space. Since the inner is almost everywhere an isometry on the unit circle, we get that the outer function is in the corresponding Hardy class. Next, by \cite[Theorem 4]{sarason1967generalized} we can factor $h_o$ as
        \[h_o=g\tilde{f}\]
        with $g\in M_{k\times l}(H^2)$,\,$\tilde{f}\in M_{l\times l'}(H^2)$ and
        $\tilde{f}^*\tilde{f}=(h_o^*h_o)^{\frac{1}{2}},\,g^*g=\tilde{f}\tilde{f}^*$ on $\bb{T}$. Note that \cite[Theorem 4]{sarason1967generalized} is for the square case. Since $h_o \in M_{k \times l'}(H^1)$, we can complete it with zeroes to a square matrix. Factor it and then take from $\tilde{f}$ only the first $l'$ columns. Moreover, we can take the inner-outer factorization of $g^t = v g'$ and replace $\tilde{f}$ by $v^t \tilde{f}$. Note that since $v$ is almost everywhere an isometry on $\bb{T}$, we have that for almost every $z \in \bb{T}$, $v(z)^t \tilde{f}(z) (v(z)^t(z) \tilde{f}(z))^* = (g'(z))^{t *} g'(z)^t$. Namely, the same condition holds. Moreover, by the definition of the norms in $H^2$, we have that $\|g\|_2^2 = \|g'\|_2^2$. Hence, if we set $f = \tilde{f} h_i$, the decomposition
        $$h=h_oh_i=g\tilde{f}h_i=g\cdot f$$
        satisfies the desired properties.
    \end{proof}
   
    \begin{thm} \label{thm:factorization}
        For every $h\in S$, there exists some $(\alpha,\beta)\in\bb{G}(k\times k)$, such that         $$h=gf^*$$ with $g\in (H^2_{\alpha,\beta})_k$, $f\in M_k(L^2)\ominus ((H^2_{\alpha,\beta})_k\cap B_FM_k(H^2))$, $g^T$ left outer and $\|g\|^2_2=\|f\|^2_2=\|h\|_1$.
    \end{thm}
    \begin{proof}
        Write $h=\overline{\varphi_\lambda(z)B_F}h'$ with $h'\in M_k(H^1)$. By Lemma \ref{thm.factorization.M_k(H^1)}, $h'$ can be factored as $gf_0$ with $f_0,\,g\in M_k(H^2)$, $g^T$ outer and $\|g\|_2^2=\|f_0\|_2^2=\|h\|_1$. Note that all matrices are square, since $h \in S$ and, therefore, $\det h$ is not identically zero.
        We first factor $h$ as
        $$h=\overline{\varphi_\lambda B_F}gf_0.$$ Taking $f=\varphi_{\lambda}(z)B_Ff_0^*$ and observing that $\|f\|_2^2=\|f_0\|_2^2$ (as both $\varphi_{\lambda}$ and $B_F$ are inner functions and $\|f_0\|^2_2=\|f_0^*\|_2^2$) finishes the factorization.
        \\\\ Now let $$\tilde{\alpha}=I_k,\, \tilde{\beta}=\frac{1}{f_{\lambda}(\lambda)}\left(g(0)^{-1}g(\lambda)-I_k\right),\,\tilde{U}=g(0)$$ and define
        $$\alpha=\left(I_k+\tilde{\beta}\tilde{\beta}^*\right)^{-\frac{1}{2}},\,\beta=\left(I_k+\tilde{\beta}\tilde{\beta}^*\right)^{-\frac{1}{2}}\tilde{\beta},\,U=\tilde{U}\left(I_k+\tilde{\beta}\tilde{\beta}^*\right)^{\frac{1}{2}}.$$
        Note that $U,\,\alpha,\,\alpha+f_{\lambda}\beta$ are invertible (as $h\in S$ and $g^t$ is left outer)  and that $\alpha\alpha^*+\beta\beta^*=I_k$ and $U\alpha=g(0),\,U\left(\alpha+f_{\lambda}(\lambda)\beta\right)=g(\lambda)$, thus $g\in (H^2_{\alpha,\beta})_k$. Lastly,
        to prove $f$ satisfies the desired property, we first observe that for every $\phi\in\mathcal{I}$,
        $$\ip{\phi g}{f}=\frac{1}{2\pi}\int_{\T}\phi gf^*=\frac{1}{2\pi}\int_{\T}\phi h=0,$$ that is $f\in (\mathcal{I}g)^{\perp}$. Therefore, it is enough to prove that $\overline{\mathcal{I}g}=(H^2_{\alpha,\beta})_k\cap B_FM_k(H^2)$. First, it is clear that $$\mathcal{I}g\subset (H^2_{\alpha,\beta})_k\cap B_FM_k(H^2).$$ 
        
        For the other direction, let $\psi\in (H^2_{\alpha,\beta})_k\cap B_FM_k(H^2)$, and write
        $$\psi(\lambda)=V(\alpha+f_{\lambda}(\lambda)\beta),\,\psi(0)=V\alpha.$$
        As $g^T$ is left-outer, $g$ is right-outer, and therefore $\overline{M_k(H^{\infty})g}=M_k(H^2)$. So there exists a sequence $\{\psi_m\}_m\subset M_k(H^{\infty})$, such that $\psi_mg\to \psi$. Write
        $$\psi_m=\psi_m(0)+\left(\psi_m(\lambda)-\psi_m(0)\right)\frac{f_{\lambda}}{f_{\lambda}(\lambda)}+B_\lambda \tilde{\psi}_m,$$ with $\tilde{\psi}_m\in M_k\left(\node\right)$.
        Therefore,
        $$\varphi_m=\psi_m -\left(\psi_m(\lambda)-\psi_m(0)\right)\frac{f_{\lambda}}{f_{\lambda}(\lambda)}\in M_k(\node).$$
        Moreover,
        $$\|\varphi_mg-\psi\|\leq \|\psi_mg-\psi\|+\left\|\frac{f_{\lambda}}{f_{\lambda}(\lambda)}\right\|\left\|\psi_m(\lambda)-\psi_m(0)\right\|\to 0,$$
        where in the last step we used that
        $$\lim_{m\to\infty}(\psi_m(\lambda)-\psi_m(0))=\psi(\lambda)g(\lambda)^{-1}-\psi(0)g(0)^{-1}=VU^{-1}-VU^{-1}=0.$$
        Lastly, we can always assume that at most one of $0$ and $\lambda$ is in $F$ (the set of points that defines the ideal $\cI$) - otherwise, we can take one out without changing $\cI$.
        Hence, by entry-wise application of Lagrange interpolation theorem, we have polynomials $\{P_{i}\}\subset \node$ such that
        $P_{i}(z_j)=\delta_{i,j}I_k$. Now consider the sequence $\varphi_m-\sum_{i=1}^n \varphi_m(z_i)P_i$. Since $\varphi_m(z_i) g(z_i) \to \psi(z_i) = 0$ for every $1 \leq i \leq n$ and $g(z_i)$ is invertible, 
        $$\|(\varphi_m-\sum_{i=1}^n \varphi_m(z_i)P_i)g-\psi\|\leq \|\varphi_mg-\psi\|+\sum_{i=1}^n\|\varphi_m(z_i)\|\|P_ig\|\to 0.$$ Moreover, $(\varphi_m-\sum_{i=1}^n\varphi_m(z_i)P_i)g\in\mathcal{I}g$, concluding the proof.
    \end{proof}
    For $(\alpha,\beta)\in\bb{G}(l'\times l)$, let 
    $$(\mathcal{M}_{\alpha,\beta})_{k,l',l}=\{Xk^{\alpha,\beta}(\cdot ,z_j) \mid X\in M_{k\times l}(\C),\,1\leq j\leq n\}.$$
    It is an immediate consequence of the reproducing property that this subspace is equal to 
    $$(H^2_{\alpha,\beta})_{k,l',l}\ominus((H^2_{\alpha,\beta})_{k,l',l}\cap B_F M_{k\times l}(H^2)).$$

        \begin{remark}
            If $l'=l=k$ we abbreviate this notation by $(\mathcal{M}_{\alpha,\beta})_k$.
        \end{remark}
    \begin{lem}
        Assume $F\in M_k(H^{\infty}_{\operatorname{node}})$ satisfies $F(z_i)=W_i$ for all $1\leq i\leq n$. Then, for $(\alpha, \beta) \in \G(k \times k)$, equation \eqref{eq.trace.condition} holds for every $(X_1,\ldots,X_n) \in (M_k(\C))^n$ if and only if the operator $P_{(\mathcal{M}_{\alpha,\beta})_k}M_F|_{(\mathcal{M}_{\alpha,\beta})_k}$ is a contraction.
            \label{lem.necessity}
    \end{lem}
    \begin{proof}
        Indeed, fix $k\in\mathbb{N}$ and let $F$ satisfy $F(z_i)=W_i$. Recall that $P_{(\mathcal{M}_{\alpha,\beta})_k}M_F|_{(\mathcal{M}_{\alpha,\beta})_k}$ is a contraction if and only if 
        $$I-P_{(\mathcal{M}_{\alpha,\beta})_k}M_F M_F^*|_{(\mathcal{M}_{\alpha,\beta})_k}\geq 0.$$
        This is true if and only if 
        \begin{equation}
\begin{aligned}
0 
&\leq \left\langle 
\left( I - P_{(\mathcal{M}_{\alpha,\beta})_k} M_F M_F^* \Big|_{(\mathcal{M}_{\alpha,\beta})_k} \right)
\left( \sum_{j=1}^n X_j\, k^{\alpha,\beta}(\cdot, z_j) \right),
\sum_{i=1}^n X_i\, k^{\alpha,\beta}(\cdot, z_i)
\right\rangle 
\end{aligned}
\end{equation}
for every $\sum_{i=1}^nX_ik^{\alpha,\beta}(\cdot,z_i)\in\mathcal{M}_{\alpha,\beta}$. However, noting that 

\begin{align*}
&\left\langle 
\left( I - P_{(\mathcal{M}_{\alpha,\beta})_k} M_F M_F^* \big|_{(\mathcal{M}_{\alpha,\beta})_k} \right)
\left( \sum_{j=1}^n X_j\, k^{\alpha,\beta}(\cdot, z_j) \right),
\sum_{i=1}^n X_i\, k^{\alpha,\beta}(\cdot, z_i)
\right\rangle \\[1ex]
&\quad= 
\left\langle 
\sum_{j=1}^n X_j\, k^{\alpha,\beta}(\cdot, z_j),
\sum_{i=1}^n X_i\, k^{\alpha,\beta}(\cdot, z_i)
\right\rangle 
-
\left\langle 
M_F^* \left( \sum_{j=1}^n X_j\, k^{\alpha,\beta}(\cdot, z_j) \right),
M_F^* \left( \sum_{i=1}^n X_i\, k^{\alpha,\beta}(\cdot, z_i) \right)
\right\rangle \\[1ex]
&\quad= 
\sum_{i,j=1}^n 
\left\langle 
X_j\, k^{\alpha,\beta}(\cdot, z_j),
X_i\, k^{\alpha,\beta}(\cdot, z_i)
\right\rangle 
- 
\sum_{i,j=1}^n 
\left\langle 
F(z_j)^* X_j\, k^{\alpha,\beta}(\cdot, z_j),
F(z_i)^* X_i\, k^{\alpha,\beta}(\cdot, z_i)
\right\rangle \\[1ex]
&\quad= 
\sum_{i,j=1}^n 
\operatorname{tr} \left(
X_i^* X_j\, k^{\alpha,\beta}(z_i, z_j)
- X_i^* F(z_i) F(z_j)^* X_j\, k^{\alpha,\beta}(z_i, z_j)
\right) \\[1ex]
&\quad= 
\operatorname{tr} \left( 
\sum_{i,j=1}^n 
\left[
X_j\, k^{\alpha,\beta}(z_i, z_j) X_i^*
- F(z_j)^* X_j\, k^{\alpha,\beta}(z_i, z_j) X_i^* F(z_i)
\right]
\right)
\end{align*}
finishes the proof. 
    \end{proof}

        We are ready to prove Theorem \ref{thm.trace.condition}.
    \begin{proof}[Proof of Theorem {\ref{thm.trace.condition}}]
    Assume $F$ satisfies $\|F\|_{\infty}\leq 1$ and $F(z_i)=W_i$ for all $1\leq i\leq n$. Since 
    $$1\geq  \|F\|_{\infty}= \|M_F\|_{op}\geq\|P_{(\mathcal{M}_{\alpha,\beta})_k}M_F|_{(\mathcal{M}_{\alpha,\beta})_k}\|_{op},$$
    the necessity part follows readily from Lemma \ref{lem.necessity}.
    In the other direction,
            let $z_1,\ldots,z_n$ be the interpolation nodes and $W_1,\ldots, W_n\in M_k(\C)$ be the interpolation data. Using Lagrange interpolation entry-wise, we obtain $Q\in M_k(\node)$ such that $Q(z_i)=W_i$.
            Since, $(\mathcal{M}_{\alpha,\beta})_k$ is $M_Q^*$ invariant,
            $$\|P_{(\mathcal{M}_{\alpha,\beta})_k}M_Q|_{(\mathcal{M}_{\alpha,\beta})_k}\|=\|M_Q^*|_{(\mathcal{M}_{\alpha,\beta})_k}\|\leq 1,$$
            where the last inequality is true by our assumption \eqref{eq.trace.condition} and Lemma \ref{lem.necessity}.
            \\\\
            Define a linear functional on $(\mathcal{I})_{\perp}$ by
            $$L_{Q}(h)=\frac{1}{2\pi}\int_{\T}\operatorname{tr}[Qh].$$ We show that $\|L_Q\|\leq 1$. By Theorem \ref{thm:factorization}, for $h \in S$, we can convert the formula of $L_Q$ to an $L^2$-inner product as follows,
            $$L_{Q}(h)=\frac{1}{2\pi}\int_{\T}\operatorname{tr}[Qh]=\frac{1}{2\pi}\int_{\T}\operatorname{tr}[Qgf^*]=\ip{Qg}{f}_{M_k(L^2)}=\ip{P_{H^2_{\alpha,\beta}}Qg}{P_{(H^2_{\alpha,\beta}\cap B_FM_k(H^2))^{\perp}}f}_{M_k(L^2)}.$$
            
            Since
            $(\mathcal{M}_{\alpha,\beta})_k=(H^2_{\alpha,\beta})_k\ominus((H^2_{\alpha,\beta})_k\cap B_FM_k(H^2))$, we can rewrite the equality above as
            
            \begin{align*}
    L_Q(h) &= \ip{P_{(\mathcal{M}_{\alpha,\beta})_k}Qg}{P_{(\mathcal{M}_{\alpha,\beta})_k}f} 
    = \\
    &= \ip{P_{(\mathcal{M}_{\alpha,\beta})_k}M_Q (P_{(\mathcal{M}_{\alpha,\beta})_k}g + P_{H^2_{\alpha,\beta}\cap B_FM_k(H^2)} g)}{P_{(\mathcal{M}_{\alpha,\beta})_k}f} = \\
    &= \ip{P_{(\mathcal{M}_{\alpha,\beta})_k}M_Q P_{(\mathcal{M}_{\alpha,\beta})_k}g}{P_{(\mathcal{M}_{\alpha,\beta})_k}f}.
\end{align*}
The last equality follows from the fact that $H^2_{\alpha,\beta}\cap B_FM_k(H^2)$ is multiplier invariant.
            By density of $S$,
            $$\|L_Q\|=\sup_{h\in S,\,\|h\|=1}|L_Q(h)|\leq\sup_{(f,g)\in A}|\ip{P_{(\mathcal{M}_{\alpha,\beta})_k}M_QP_{(\mathcal{M}_{\alpha,\beta})_k}g}{P_{(\mathcal{M}_{\alpha,\beta})_k}f}|,$$
            where
            $A=\{(f,g) \mid \|f\|=1,\,\|g\|=1,g\in (H^2_{\alpha,\beta})_k,\,f\in ((H^2_{\alpha,\beta})_k\cap B_FM_k(H^2))^{\perp}\}$.
            Lastly, by the Cauchy-Schwarz inequality,
            $$\sup_{(f,g)\in A}|\ip{P_{(\mathcal{M}_{\alpha,\beta})_k}M_QP_{(\mathcal{M}_{\alpha,\beta})_k}g}{P_{(\mathcal{M}_{\alpha,\beta})_k}f}| \leq \|P_{(\mathcal{M}_{\alpha,\beta})_k}M_Q|_{(\mathcal{M}_{\alpha,\beta})_k}\|\leq 1.$$ By the Hahn-Banach Theorem, we can extend $L_Q$ to a contractive functional on $M_k(L^1)$. Now the Riesz representation theorem implies that such a functional has the form of $L_F$ for some $F\in M_k(L^{\infty})$, where
            $$L_F(h)=\frac{1}{2\pi}\int_{\T}\operatorname{tr}[F(z)h(z)]|dz|$$ and $1\geq \|L_F\|=\|F\|_{\infty}$. As $L_F$ is an extension of $L_Q$, $L_{F-Q}|_{(\mathcal{I})_{\perp}}=L_F|_{(\mathcal{I})_{\perp}}-L_Q=0$, that is, $F-Q\in ((\mathcal{I})_{\perp})^{\perp}=\mathcal{I}$. As $Q\in M_k(\node)$ and $\mathcal{I}\subset M_k(\node)$, $F\in M_k(\node)$ and $F(z_i)=Q(z_i)=W_i$. We noted that $\|F\|_{\infty}\leq 1$, implying $F$ is the desired solution.
        \end{proof}
        \begin{cor}(Scalar case)
            Given $\{z_1,\ldots,z_n\}\subset\bb{D},\,\{w_1,\ldots,w_n\}\subset\C$, there exists $f\in \node$ such that $f(z_i)=w_i$ if and only if the Pick matrix
            $$\left[(1-w_i\overline{w_j})k^{\alpha,\beta}(z_i,z_j)\right]_{i,j=1}^n$$ is positive semi-definite for all $\alpha,\beta\in\G(1\times 1)$.
        \end{cor}
            \begin{proof}
                By what we proved, such a function exists if and only if for every $(\alpha,\beta)\in\G(1\times 1)$ and for every choice of scalars $(x_1,\ldots,x_n)\in\C^n$,
                $$\sum_{i,j=1}^n (1-w_i\overline{w_j})k^{\alpha,\beta}(z_i,z_j)x_i\overline{x_j}\geq 0.$$ This condition is exactly the condition
                $$\ip{\left[(1-w_i\overline{w_j})k^{\alpha,\beta}(z_i,z_j)\right]_{i=1}^n\begin{pmatrix}
                x_1\\\vdots\\x_n
            \end{pmatrix}}{\begin{pmatrix}
            x_1\\\vdots\\x_n
        \end{pmatrix}}
        \geq 0$$ for every $\begin{pmatrix}
        x_1\\\vdots\\x_n
    \end{pmatrix}\in\C^n$, which is exactly the condition that
    $$\left[(1-w_i\overline{w_j})k^{\alpha,\beta}(z_i,z_j)\right]_{i=1}^n\geq 0.$$
    \end{proof} 
We conclude this section with two important corollaries that connect the norm in $\node/I$ and the norm of compression operators.

\begin{cor}(Distance formula)
    Let $\mathcal{I}$ be as before. Then, for every $F\in M_k(\node)$,
    $$dist(F,\mathcal{I})=\|F+\mathcal{I}\|=\sup_{{(\alpha,\beta)}\in\G(k\times k)}\|M_F^*|_{\mathcal{M_{\alpha,\beta}}}\|.$$
    \label{col.distance_formula}
\end{cor}
\begin{proof}
    By the definition of the quotient norm,
    $$\|F+\mathcal{I}\|=\inf\{\|G\| \mid G\in M_k(\node),\,F(z_i)=G(z_i)=W_i,\,\forall i=1,\ldots,n\}.$$ By rescalling, it follows that there exists a $G\in M_k(\node)$, such that $G(z_i)=W_i$ and $\|G\|\leq M$ if and only if
    \begin{equation}
        \sum_{i,j=1}^n \operatorname{tr}\left[M^2X_jk^{\alpha,\beta}(z_i,z_j)X_i^*-W_j^*X_jk^{\alpha,\beta}(z_i,z_j)X_i^*W_i\right]\geq 0,\label{eq:100} \end{equation}
        for all $n$-tuples $(X_1,\ldots,X_n)\in (M_k(\C))^n$ and $(\alpha,\beta)\in\G(k\times k)$. Therefore,
        $$dist(F,\mathcal{I})=\inf\{M \mid \eqref{eq:100}\,\text{is satisfied}\,\forall (\alpha,\beta)\in\G(k\times k),\,(X_1,\ldots,X_n)\in (M_k(\C))^n\}=$$$$=\inf\{M \mid \|M_F^*|_{\mathcal{M_{\alpha,\beta}}}\|\leq M,\,\forall (\alpha,\beta)\in\G(k\times k)\}.$$
        Clearly,
        $$\sup_{(\alpha,\beta)\in\G(k\times k)}\|M_F^*|_{\mathcal{M_{\alpha,\beta}}}\|=\inf\{M \mid \|M_F^*|_{\mathcal{M_{\alpha,\beta}}}\|\leq M,\,\forall (\alpha,\beta)\in\G(k\times k)\}.$$
    \end{proof}
    For fixed $m, r\in\N$ and $(\alpha,\beta)\in\G(m\times m)$, define a map
    $\varphi_{\alpha,\beta} \colon M_r(\node)\to B((\mathcal{M}_{\alpha,\beta})_{r,m})$
    by $$\varphi_{\alpha,\beta}(F)=P_{(\mathcal{M}_{\alpha,\beta})_{r,m}}M_F|_{(\mathcal{M}_{\alpha,\beta})_{r,m}}.$$
    By Sarason's Lemma \cite[Lemma 0]{sarason1965spectral}, this map is an algebra homomorphism, and it is straightforward that $\ker(\varphi)=\mathcal{I}$.
    \begin{cor}
    (Generalized distance formula)
    Let $r\in\N$ and $F\in M_r\left(\node\right)$. Denote $M_r(I)=\mathcal{I}$. Then,
    $$\left\|F+\mathcal{I}\right\|= \sup_m \sup\left\{\left\|P_{(\mathcal{M_{\alpha,\beta}})_{r,m}} M_F|_{(\mathcal{M_{\alpha,\beta}})_{r,m}}\right\| \mid (\alpha,\beta) \in \mathbb{G}(m\times m)\right\}.$$
\end{cor}
\begin{proof}
    The fact that the RHS is greater than or equal to the LHS is clear by Corollary \ref{col.distance_formula}.
    Fix $m\in\N$ and $(\alpha,\beta)\in\G(m\times m)$, and consider $\varphi$ as before. Clearly, $\varphi_{\alpha,\beta}$ is a contraction, as
    $$\left\|\varphi_{\alpha,\beta}(F)\right\|=\left\|P_{(\mathcal{M_{\alpha,\beta}})_{r,m}}M_F|_{(\mathcal{M_{\alpha,\beta}})_{r,m}}\right\|\leq \left\|P_{(\mathcal{M_{\alpha,\beta}})_{r,m}}\right\|\left\|M_F|_{H^2_{\alpha,\beta}}\right\|\leq \left\|M_F\right\|=\left\|F\right\|_{\infty},$$ 
    where we used the fact that as a multiplier on $H^2_{\alpha,\beta}$, $M_F|_{H^2_{\alpha,\beta}}$ is the compression of $M_F\in \operatorname{Mult}(M_k(H^2))$ to the invariant subspace $H^2_{\alpha,\beta}$, and that
    $\|M_F\|_{op}=\|F\|_{\infty}$ for every $F\in M_{k}(H^{\infty})$.
    Lastly, note that for every $F\in M_r(\node),\,G\in\mathcal{I}$,
    $$\|\varphi(F)\|=\|\varphi(F+G)\|\leq \|F+G\|.$$ Hence, 
    $$\|P_{(\mathcal{M_{\alpha,\beta}})_{r,m}}M_F|_{(\mathcal{M_{\alpha,\beta}})_{r,m}}\|\leq\inf_{G\in\mathcal{I}}\|F+G\|=\|F+I\|.$$ As this is independent of $m$ and $(\alpha,\beta)$,
    $$\sup_m \sup \left\{\left\|P_{(\mathcal{M_{\alpha,\beta}})_{r,m}} M_F|_{(\mathcal{M_{\alpha,\beta}})_{r,m}}\right\| \mid (\alpha,\beta) \in \mathbb{G}(m\times m)\right\}\leq\|F+I\|,$$
    as desired.
\end{proof}

\section{Boundary representations from scalar interpolation}\label{sec:boundary}
We begin this section with some results on representations of $\node/I$ for general $0\neq\lambda\in\D$ and $\{z_1,\ldots,z_n\}\subset\D\setminus\{0,\lambda\}$. Here $I$ is again the ideal of functions that vanish on the $z_i$. By \cite[Proposition 2.3.4]{BlecherLeMerdy}, $\node/I$ is an operator algebra when equipped with the quotient operator space structure. By a representation of an operator algebra on a Hilbert space $\cH$, we mean a unital completely contractive homorphism to $B(\cH)$. The spectral structure of $\node/I$ is quite simple and is described in the following proposition.

\begin{lem} \label{lem:diagonalizable} 
            The algebra $\node/I$ is generated by $h+I$, with $h\in \node$ being the degree $n+1$ polynomial satisfying $h(z_i)=z_i$ for all $i=1,\ldots,n$ and $h(0)=h(\lambda)=0$. The spectrum of $\node/I$ consists of the characters of evaluation at $z_1,\ldots,z_n$. Therefore, for every finite-dimensional representation $\varphi \colon \node/I \to B(\cH)$ and every $f \in \node$, $\varphi(f + I) = f(\varphi(h + I))$ is diagonalizable.
\end{lem}
\begin{proof}
    Let $f\in \node$. There exists a polynomial $p(z)=\sum_{k=0}^{n+1}a_kz^k\in \node$ satisfying 
    $p(z_i)=f(z_i)$ for all $1\leq i\leq n$. Define $\tilde{f}(z)=\sum_{k=0}^{n+1}a_kh^k\in \node$. Then,
    $$\tilde{f}(z_i)=\sum_{k=0}^{n+1}a_kh^k(z_i)=\sum_{k=0}^{n+1}a_kz_i^k=p(z_i)=f(z_i).$$
    So $f+I=\tilde{f}+I=\sum_{k=0}^{n+1}a_kh^k+I=\sum_{k=0}^{n+1}a_k(h+I)^k$. We conclude that every character of $\node/I$ is determined on $h$. It is also easy to see that for every $\alpha \in \C \setminus \{z_1,\ldots,z_n\}$, $h - \alpha + I$ is invertible in $\node/I$. The claim on the spectrum now follows.

    For the claim that for every $f \in \node$, $\varphi(f + I) = f(\varphi(h + I))$, we use a density argument. Firstly, since $I$ is weak$^*$ closed, the quotient map
            $$\pi \colon \node\to\node/I$$ 
            is weak$^*$ continuous. Therefore, $\varphi\circ\pi$ is weak$^*$ continuous. Since polynomials are weak$^*$ dense, it is enough to prove it for polynomials. However, this follows from the fact that $\varphi$ is an algebra homomorphism.
            Now consider $p(z)=\prod_{i=1}^n (z-z_i)\cdot z(z-\lambda)\in H^{\infty}_{\text{node}}$. Then, $p+I=0+I$. Since $\varphi$ is a representation, $\varphi(p+I)=0$. On the other hand, $p+I=\prod_{i=1}^n(h-z_i)\cdot h(h-\lambda)+I$, and since $\varphi$ is a representation,
            $$\varphi(p+I)=\prod_{i=1}^n\left(\varphi(h+I)-z_iI_d\right)\varphi(h+I)\left(\varphi(h+I)-\lambda I_d\right)=0.$$
            Therefore, the minimal polynomial of $\varphi(h+I)$ factors into different linear factors, i.e., $\varphi(h+I)$ is diagonalizable. The general claim follows.
        \end{proof}

        \begin{lem}
    For $(\alpha,\beta)\in \G(1\times 1)$, the map
    $$\rho_{\alpha,\beta} \colon \node/I\to B(\mathcal{M_{\alpha,\beta}})$$ given by $f+I\to P_{\mathcal{M_{\alpha,\beta}}}M_f|_{\mathcal{M_{\alpha,\beta}}}$ is an operator algebra representation.
\end{lem}
\begin{proof}
    The claim follows at once from the Sarason Lemma \cite[Lemma 0]{sarason1965spectral}.
\end{proof}
\begin{remark}

            For every $(\alpha,\beta)$ on the unit sphere of $\C^2$, let $T = \rho_{\alpha,\beta}(h+I)$. Then, the minimal polynomial of $T$ is precisely, $p(z)=\prod_{i=1}^n(z-z_i)$. This is since $T^*$ is an operator on $n$ dimensional space with eigenvalues $\overline{z_1},\ldots,\overline{z_n}$.
            \label{remark.minimal.polynomial}
        \end{remark}

Now we need to recall some notations and terminology from the theory of $C^*$-covers and boundary representations of operator algebras. Recall, that a $C^*$-cover of an operator algebra $\cA$ is a $C^*$-algebra $B$ together with a completely isometric homomorphisms $\iota \colon \cA \to B$, such that $B = C^*(\iota(\cA))$. For every operator algebra $\cA$, there are two extremal $C^*$-covers. Let $C^*_{\max}(\cA)$ denote its maximal $C^*$-cover. The maximal $C^*$-cover for an operator system was constructed by Kirchberg and Wassermann \cite{KirchWass}. The maximal $C^*$-cover of an operator algebra is described in \cite[Proposition 2.4.2]{BlecherLeMerdy}. The maximality of this $C^*$-cover is encoded in its universal property. For every $C^*$-algebra $B$ and a completely contractive homomorphism $\varphi \colon \cA \to B$, there exists a $*$-homomorphism $\pi \colon C^*_{max}(\cA) \to B$ extending $\varphi$. The minimal $C^*$-cover of $\cA$ is the $C^*$-envelope that we will denote by $C^*_e(\cA)$. The existence of the $C^*$-envelope was conjectured by Arveson \cite{arveson1969subalgebras} and proved by Hamana \cite{hamana1979injective}. The $C^*$-envelope of $\cA$ has the universal property that for every $C^*$-cover $\iota \colon \cA \to B$, there exists a surjective homomorphism $\pi \colon B \to C^*_e(\cA)$, such that $\pi \circ \iota = \iota_e$, where $\iota_e$ stands for the completely isometric embedding of $\cA$ in $C^*_e(\cA)$.
    \begin{dfn}
        Let $\mathcal{H}\subset \mathcal{K}$ be Hilbert spaces, $\mathcal{A}$ an operator algebra and
        $$\psi_1\colon\mathcal{A}\to B(\mathcal{H}),\,\psi_2\colon\mathcal{A}\to B(\mathcal{K})$$ be two representations. We say that $\psi_2$ is \textit{dilation} of $\psi_1$, if $P_{\mathcal{H}}\psi_2|_{\mathcal{H}}=\psi_1$. Moreover, $\psi_1$ is called \textit{dilation maximal} if every dilation of $\psi_1$ is trivial, i.e., if $\psi_2$ dilates $\psi_1$, then $\psi_2=\psi_1\oplus\psi$ for some representation $\psi$.
    \end{dfn}
    It follows from Sarason's lemma that $\cH$ is semi-invariant for $\psi_2(A)$. The following definition is due to Arveson \cite{arveson1969subalgebras}.
    \begin{dfn}
        Let $\mathcal{A}$ be an operator algebra. A representation
        $$\pi\colon C^*(\mathcal{S})\to B(\mathcal{H})$$ is said to have the unique extension property if for every unital completely positive map (ucp for short) $\varphi \colon C^*(\mathcal{S})\to B(\mathcal{H})$, such that $\pi|_{\mathcal{A}}=\varphi|_{\mathcal{A}}$, we have that $\pi=\varphi$.
    \end{dfn}
    \begin{dfn}
        Let $(B,\iota)$ be a $C^*$-cover of an operator algebra $\mathcal{A}$. A representation
        $$\psi\colon B \to B(\mathcal{H})$$ is said to be a boundary representation, if it is irreducible and has the unique extension property.
    \end{dfn}
    It follows from results of Muhly and Solel \cite{MuhlySolel-char} (see also \cite{DritMcC-boundary} and \cite{arveson-nc_choquet_I}) that $\pi$ has the unique extension property if and only if $\pi|_{\cA}$ is dilation maximal. The following theorem is a particular case of \cite[Proposition 6.5.1]{davidson2019noncommutative}.
    \begin{thm}
        Let $\mathcal{A}$ be an operator algebra and $\psi \colon C^*_{\max}(\mathcal{A})\to B(\mathcal{H})$ be a boundary representation of $\mathcal{A}$. Let $\pi \colon C^*_{\max}(\mathcal{A})\to C^*_e(\mathcal{A})$ be the surjective $*$-homomorphism guaranteed by the universal properties. Then, there is a representation
        $$\sigma \colon C^*_e(\mathcal{A})\to B(\mathcal{H})$$
        such that $\sigma\circ\pi=\psi$ (that is, $\psi$ factors through $C^*_e(\mathcal{A})$).\label{thm.boundary.rep.of.C^*_e}
    \end{thm}
    Before presenting the main result of this section, we need two technical lemmas.
        \begin{lem}
        Fix $0\neq \lambda\in \bb{D}$ and let $\omega\in\bb{D}$.
        \begin{enumerate}
            \item If there exist $\zeta\neq\omega \in \D$, such that
        $B_{\lambda}(\zeta)=B_{\lambda}(\omega)$,
        then
        $$B_{\omega,\zeta}(z):=\frac{z-\zeta}{1-\overline{\zeta}z}\cdot\frac{z-\omega}{1-\overline{\omega}z}\in \node.$$
        \item If $B_{\lambda}'(\omega)=0$, then
        $$B_{\omega}^2(z)=\left(\frac{z-\omega}{1-\overline{\omega}z}\right)^2\in \node.$$
        \end{enumerate}
        \label{lem:Blaschke.diff.points}
    \end{lem}
    \begin{proof}
        Let $\varphi\in \operatorname{Aut}(\bb{D})$ be such that $\varphi\left(B_{\lambda}(\omega)\right)=0$. 
        \begin{enumerate}
            \item 
        In the first case,
        $\varphi\circ B_{\lambda}$ is a Blaschke product with 
        $2$ zeroes that vanishes at $\zeta$ and $\omega$, hence
        $$\varphi\circ B_{\lambda}=e^{i\theta}\frac{z-\zeta}{1-\overline{\zeta}z}\cdot\frac{z-\omega}{1-\overline{\omega}z}\in \node.$$
        \item In the second case, by the chain rule
        $$\left(\varphi\circ B_{\lambda}\right)'(\omega)=\varphi'\left(B_{\lambda}(\omega)\right)B'_{\lambda}(\omega)=0,$$
        i.e., $\varphi\circ B_{\lambda}$ is a Blaschke product with at most $2$ zeroes with a double zero at $\omega$. Therefore
        $$\varphi\circ B_{\lambda}(z)=e^{i\theta}B_\omega^2\in \node.$$
        \end{enumerate}      
    \end{proof}

    \begin{lem} 
        Let $0\neq \lambda \in \D$. Let $z_1,\ldots,z_4\in\bb{D}\setminus\{0,\lambda\}$ be such that $B_{\lambda}(z_1) = B_{\lambda}(z_4)$ and $B_{\lambda}(z_2) = B_{\lambda}(z_3)$. Assume, furthermore, that there exists $\omega \in \D$ such that $B_{z_1,z_2,\omega} \in \node$.
        Then, there exists some $\zeta\in\bb{D}$ such that
        $$B_{z_3,z_4,\zeta}(z)=\frac{z-z_3}{1-\overline{z_3}z}\cdot\frac{z-z_4}{1-\overline{z_4}z}\cdot\frac{z-\zeta}{1-\overline{\zeta}z}\in \node.$$
        Moreover, if $z_1,\ldots,z_4\in\bb{R}$ and $\lambda \in \R$, then $\omega,\,\zeta\in\bb{R}$.\label{lem:existence.of.point}
    \end{lem}   
    \begin{proof}
        By Lemma \ref{lem:Blaschke.diff.points} and our assumption,
        $$B_{z_1,z_4},\,B_{z_2,z_3}\in \node.$$
        If $B_{\lambda}'(\omega)=0$, then $\left(\frac{z-\omega}{1-\overline{\omega}z}\right)^2\in \node$. However,
        $$B_{z_3,z_4,\omega}(0)=\frac{B_{z_1,z_2}(0)\cdot B_{z_3,z_4}(0)\cdot B_{\omega}^2(0)}{B_{z_1,z_2,\omega}(0)}=\frac{B_{z_1,z_2}(\lambda)\cdot B_{z_3,z_4}(\lambda)\cdot B_{\omega}^2(\lambda)}{B_{z_1,z_2,\omega}(\lambda)}=B_{z_3,z_4,\omega}(\lambda).$$
        Here we used the fact that
        $$B_{z_1,z_2}\cdot B_{z_3,z_4}=B_{z_1,z_4}\cdot B_{z_2,z_3}.$$
        Otherwise, since $B_{\lambda}$ is $2$ to $1$, there exists some $\zeta\neq\omega$ such that $B_{\lambda}(\zeta)=B_{\lambda}(\omega)$. By Lemma \ref{lem:Blaschke.diff.points}, $B_{\zeta,\omega}\in \node$. By similar arguments,
        $$B_{z_3,z_4,\zeta}(0)=\frac{B_{z_1,z_2}(0)\cdot B_{z_3,z_4}(0)\cdot B_{\zeta,\omega}(0)}{B_{z_1,z_2,\omega}(0)}=\frac{B_{z_1,z_2}(\lambda)\cdot B_{z_3,z_4}(\lambda)\cdot B_{\zeta,\omega}(\lambda)}{B_{z_1,z_2,\omega}(\lambda)}=B_{z_3,z_4,\zeta}(\lambda).$$
        Lastly, assume $z_1,\ldots,z_4, \lambda$ are real. By plugging $0$ and $\lambda$ into $B_{z_1,z_2,\omega}$ we obtain the equality
        $$-z_1z_2\omega=\frac{\lambda-z_1}{1-z_1z}\cdot\frac{\lambda-z_2}{1-z_2z}\cdot\frac{\lambda-\omega}{1-\overline{\omega}\lambda}.$$
        Therefore,
        $$\omega\cdot\frac{1-\overline{\omega}\lambda}{\lambda-\omega}=\frac{\omega-|\omega|^2\lambda}{\lambda-\omega}\in\bb{R}.$$
        Consequently, there exists $c\in\bb{R}$ such that
        $$(c+1)\omega=\lambda(c+|\omega|^2)\in\bb{R}.$$
        Since $\lambda \neq 0$ and $|\omega| < 1$. we must have $c \neq -1$. Therefore, we conclude that $\omega\in\bb{R}$.
        A similar argument shows that $\zeta\in\bb{R}$.
    \end{proof}
 From here on out, $0\neq\lambda\in\bb{R}\cap\bb{D}$ is fixed. The following definition encodes the technical properties that we need to ensure the abundance of boundary representations.
        \begin{dfn} \label{dfn:big}
            Four points $z_1,\ldots,z_4\in\bb{D}\cap\bb{R}\setminus\left\{0,\lambda\right\}$ will be called \textbf{good points} with respect to $\lambda$, if there exists an infinite family $\mathcal{A}\subset\bb{T}\subset\bb{G}(1\times 1)$ such that the following conditions hold:
            \begin{enumerate}
                \item $B_{\lambda}(z_1)=B_{\lambda}(z_4),\,B_{\lambda}(z_2)=B_{\lambda}(z_3)$.
                \item There exists $\omega\in\bb{D}\cap\bb{R}$ such that
                $$B_{z_1,z_2,\omega}(z):=\frac{z-z_1}{1-\overline{z_1}z}\cdot\frac{z-z_2}{1-\overline{z_2}z}\cdot\frac{z-\omega}{1-\overline{\omega}z}\in \node.$$
                    \item For every $(\alpha,\beta)\in\mathcal{A}$, the matrix
                    $$\left[k^{\alpha,\beta}(z_i,z_j)\right]_{i,j=1}^4$$
                    has a column (equivalently, row) with no zeroes.
                 
                    \item  For every $(\alpha,\beta)\neq (\alpha',\beta')\in \mathcal{A}$ there exists a pair $z_i\neq z_j$, such that
                    \begin{equation}
        \left|\frac{k^{\alpha,\beta}(z_j,z_i)}{\sqrt{k^{\alpha,\beta}(z_i,z_i)\cdot k^{\alpha,\beta}(z_j,z_j)}}\right|^2\neq \left|\frac{k^{\alpha',\beta'}(z_j,z_i)}{\sqrt{k^{\alpha',\beta'}(z_i,z_i)\cdot k^{\alpha',\beta'}(z_j,z_j)}}\right|^2.\label{eq:0}
             \end{equation}
            
        \item For every $(\alpha,\beta)\in\mathcal{A}$,
        $$\begin{cases}
            k_{z_{\ell}}\not\in \operatorname{span}\{k^{\alpha,\beta}_{z_3},k^{\alpha,\beta}_{z_4},\,P_{\mathcal{M}_{\alpha,\beta}}k_{\zeta}^{\alpha,\beta}\}&\ell\in\{1,2\}
            \\
            k_{z_{\ell}}\not\in\operatorname{span}\{k^{\alpha,\beta}_{z_1},k^{\alpha,\beta}_{z_2},\,P_{\mathcal{M}_{\alpha,\beta}}k_{\omega}^{\alpha,\beta}\}&\ell\in\{3,4\}
        \end{cases},\,\begin{cases}
            f_{\ell}\not\in\operatorname{span}\{k^{\alpha,\beta}_{z_1},k^{\alpha,\beta}_{z_2},\,P_{\mathcal{M}_{\alpha,\beta}}k_{\omega}^{\alpha,\beta}\}&\ell\in\{1,2\}
            \\
            f_{\ell}\not\in\operatorname{span}\{k^{\alpha,\beta}_{z_3},k^{\alpha,\beta}_{z_4},\,P_{\mathcal{M}_{\alpha,\beta}}k_{\zeta}^{\alpha,\beta}\}&\ell\in\{3,4\}
        \end{cases},$$
        where $\zeta$ is some point satisfying $B_{z_3,z_4,\zeta}\in \node$ that exists by Lemma \ref{lem:existence.of.point}, and $\{f_1,\ldots,f_4\}\subset \cM_{\alpha,\beta}$ is the dual basis of $B=\{k^{\alpha,\beta}_{z_i} \mid 1\leq i\leq 4\}$ (i.e., $f_i(z_j)=\delta_{i,j}$). 
        \end{enumerate}
        \end{dfn}
        \begin{thm} \label{thm:good_then_big_envelope}
            Let $z_1,\ldots,z_4 \in \D \setminus\{0,\lambda\}$ be good points and let $I \subset \node$ be the ideal of functions that vanish on $z_1,\ldots,z_4$. Then, the representations $\sigma_{\alpha,\beta} \colon C^*_{max}(\node/I) \to B(\cM_{\alpha,\beta})$ that extend the representations $\rho_{\alpha,\beta}$ for $(\alpha, \beta) \in\mathcal{A}$ form an infinite family of unitarily inequivalent boundary representations. In particular, $C^*_e(\node/I)$ is infinite-dimensional.
            \label{thm:good.points}
        \end{thm}
        We break the proof of Theorem \ref{thm:good.points} into two parts. First we show that for every $(\alpha,\beta),\,(\alpha',\beta')\in \mathcal{A}$, the representations $\rho_{\alpha,\beta}$ and $\rho_{\alpha',\beta'}$ are unitarily inequivalent and irreducible. We then proceed to show that they are boundary representations for $\node$.

        \begin{prop}
            The family of representations $(\sigma_{\alpha,\beta})_{(\alpha,\beta)\in\mathcal{A}}$ of $C^*_{max}(\node/I)$ 
            is a family of irreducible unitarily inequivalent representations. 
        \end{prop}
        \begin{proof}
            Fix $(\alpha,\beta)\in\mathcal{A}$. Recall that in the scalar case, a basis for $\mathcal{M_{\alpha,\beta}}$ is
    $$B=\left\{k^{\alpha,\beta}_{z_1},k^{\alpha,\beta}_{z_2},k^{\alpha,\beta}_{z_3},k^{\alpha,\beta}_{z_4}\right\}.$$
    For every $f\in \node$, by Lemma \ref{lem.eigenvalues.of.multiplier}, $M_f^*$ is diagonal with respect to $B$. Let $G$ be the Gram matrix of $B$, namely,
$$(G)_{i,j}=\ip{k^{\alpha,\beta}_{z_j}}{k^{\alpha,\beta}_{z_i}}=k^{\alpha,\beta}(z_i,z_j).$$
Then,
$$[P_{\mathcal{M_{\alpha,\beta}}}M_f|_{\mathcal{M_{\alpha,\beta}}}]_{C}=\left([M_f^*|_{\mathcal{M}_{\alpha,\beta}}]_{C}\right)^*=G^{-1/2}\begin{pmatrix}
f(z_1)&0&0&0
\\
0&f(z_2)&0&0
\\
0&0&f(z_3)&0
\\
0&0&0&f(z_4)
\end{pmatrix}G^{1/2}=G^{-\frac{1}{2}}D_fG^{\frac{1}{2}},$$
where $C$ is the orthonormal basis $\left\{\sum_{j=1}^4 (G^{-\frac{1}{2}})_{j,i}k^{\alpha,\beta}_{z_j} \mid 1\leq i\leq 4\right\}$.
\\
\hspace{0.25cm}

Since $B(\mathcal{M_{\alpha,\beta}})\cong M_4(\bb{C})$ as $C^*$-algebras via $T\to [T]_{C}$, it is enough to show that
$$C^*(\rho_{\alpha,\beta}(\node/I)) \cong C^*(\underbrace{\left\{[P_{\mathcal{M_{\alpha,\beta}}}M_f|_{\mathcal{M_{\alpha,\beta}}}]_{C} \mid f\in \node\right\}}_{\mathcal{S}})=M_4(\bb{C}).$$
To see that, for $1\leq i\leq 4$ let $f_i\in \node$ be such that $f_i(z_j)=\delta_{i,j}$ and let
$$A_i= P_{\mathcal{M_{\alpha,\beta}}}M_{f_i}|_{\mathcal{M_{\alpha,\beta}}} =  G^{-\frac{1}{2}}E_{ii} G^{\frac{1}{2}}\in C^*(\rho_{\alpha,\beta}(\node/I)).$$
Then, $$\sum_{i=1}^4 \frac{1}{k^{\alpha,\beta}(z_i,z_i)}A_iA_i^*=G^{-1}\in C^*(\mathcal{S}).$$
As $C^*$ algebras are closed under square-roots, $$G^{1/2}\in C^*(\mathcal{S}).$$
Moreover, note that,
$$A_iA_j^*=G^{-\frac{1}{2}}D_{f_i}GD_{f_j}G^{-\frac{1}{2}}=k^{\alpha,\beta}(z_i,z_j)G^{-\frac{1}{2}}E_{i,j}G^{-\frac{1}{2}}.$$
By $(3)$ in Definition \ref{dfn:big}, we can assume without loss of generality that $$k^{\alpha,\beta}(z_1,z_2),\,k^{\alpha,\beta}(z_1,z_3),\,k^{\alpha,\beta}(z_1,z_4)\neq 0.$$
Hence, 
$\frac{1}{k^{\alpha,\beta}(z_i,z_j)}G^{\frac{1}{2}}A_iA_j^*G^{\frac{1}{2}}=G^{1/2}G^{-1/2}E_{i,j}G^{-1/2}G^{1/2}=E_{i,j}\in C^*(\rho_{\alpha,\beta}(\node/I))$
for $1\leq i=j\leq 4$ and $i=1,\,j=2,3,4$. Consequently,
$$E_{2,3}=E_{2,1}E_{1,3},\,E_{2,4}=E_{2,1}E_{1,4},\,E_{3,4}=E_{3,1}E_{1,4}\in C^*(\rho_{\alpha,\beta}(\node/I)),$$ i.e., $\left\{E_{i,j} \mid 1\leq i,j\leq 4\right\}\subset C^*(\rho_{\alpha,\beta}(\node/I))$, concluding the proof of irreducibility. For the second part, recall that
$\sigma_{\alpha,\beta}$ and $\sigma_{\alpha',\beta'}$ are unitarily equivalent if there exists a unitary  $U \colon \mathcal{M_{\alpha,\beta}}\to M_{\alpha',\beta'}$, such that for every $f+I\in \node/I$,
    $$U\rho_{\alpha,\beta}(f+I)^*=\rho_{\alpha',\beta'}(f+I)^*U.$$ Equivalently, there exists some unitary $U$ such that for all $f+I\in \node/I$, 
$$UM_f^*|_{\mathcal{M_{\alpha,\beta}}}=M_f^*|_{\mathcal{M}_{\alpha',\beta'}}U.$$
    In particular, if $f_i$ is the polynomial satisfying $f_i(z_j)=\delta_{i,j}$, then for every $1\leq i\leq 4$
    $$f_i(z_i)Uk_{z_i}^{\alpha,\beta}=Uk^{\alpha,\beta}_{z_i}=UM_{f_i}^*(k^{\alpha,\beta}_{z_i})=M_{f_i}^*(Uk^{\alpha,\beta}_{z_i}).$$ 
    Therefore, $Uk^{\alpha,\beta}_{z_i}\neq 0$ is an eigenvector of $M_{f_i}^*$ as an operator on $\mathcal{M_{\alpha',\beta'}}$ with eigenvalue 
    $1$. By Lemma \ref{lem.eigenvalues.of.multiplier}, $Uk^{\alpha,\beta}_{z_i}\in \operatorname{\operatorname{span}}\left\{k^{\alpha',\beta'}_{z_i}\right\}$ (one dimensional eigenspace), i.e., $Uk^{\alpha,\beta}_{z_i}=\gamma_i k^{\alpha',\beta'}_{z_i}$. Moreover, as $U$ is unitary,
$$\|k^{\alpha,\beta}_{z_i}\|=|\gamma_i|\|k^{\alpha',\beta'}_{z_i}\|$$
    and
    $$\ip{k^{\alpha,\beta}_{z_i}}{k^{\alpha,\beta}_{z_j}}=\gamma_i\overline{\gamma_j}\ip{k^{\alpha',\beta'}_{z_i}}{k^{\alpha',\beta'}_{z_j}}.$$ Taking the absolute value on both sides, we get
    $$\left|\ip{k^{\alpha,\beta}_{z_i}}{k^{\alpha,\beta}_{z_j}}\right|=\frac{\|k^{\alpha,\beta}_{z_i}\|}{\|k^{\alpha',\beta'}_{z_i}\|}\frac{\|k^{\alpha,\beta}_{z_j}\|}{\|k^{\alpha',\beta'}_{z_j}\|}\left|\ip{k^{\alpha',\beta'}_{z_i}}{k^{\alpha',\beta'}_{z_j}}\right|.$$ Equivalently,
    \begin{equation}
        \frac{|k^{\alpha,\beta}(z_j,z_i)|^2}{k^{\alpha,\beta}(z_i,z_i)\cdot k^{\alpha,\beta}(z_j,z_j)}=\frac{|k^{\alpha',\beta'}(z_j,z_i)|^2}{k^{\alpha',\beta'}(z_i,z_i)\cdot k^{\alpha',\beta'}(z_j,z_j)}
    \end{equation}
    for every $z_i,\,z_j$, contradicting $(4)$ in Definition \ref{dfn:big}. 
        \end{proof}

    Let $\{f_1,\ldots,f_4\} \subset \cM_{\alpha,\beta}$ be the dual basis of $\{k^{\alpha,\beta}_{z_1},\ldots,k^{\alpha,\beta}_{z_4}\}$ (i.e., $f_i(z_j)=\delta_{i,j}$). We will show that $\rho_{\alpha,\beta}$ is dilation maximal for every $(\alpha,\beta)\in \mathcal{A}$.
    Our strategy is to prove that every $(\alpha,\beta)\in\mathcal{A}$, $\rho_{\alpha,\beta}$ has no non-trivial extensions and no non-trivial coextensions. This condition is known to be equivalent to dilation maximality. This approach is an adaptation of the approach taken by McCullough \cite{mccullough2001isometric}. Due to the existence of the generator $h+I \in \node/I$, for every representation, an invariant/coinvariant/reducing subspace for the image of $h+I$ is such for the image of the entire algebra. This phenomenon distinguishes this case from the case when the set of interpolation nodes contains the constraints.
    
\begin{thm}
    For every $(\alpha, \beta) \in\mathcal{A}$, $\rho_{\alpha,\beta}$ is co-extension maximal.\label{thm.no.non.trivial.coextensions}
\end{thm}

Fix the permutation $\tau=(14)(23)\in S_4$. We begin by proving several preliminary results.

\begin{lem}\label{rmrk:0}
    Let $\varphi \in \node$ be a Blaschke product with zeroes $x_1,\ldots,x_n \subset \D \setminus \{0,\lambda\}$. Set $$K_{\varphi} = \operatorname{span}\left\{k^{\alpha,\beta}_{x_1},\ldots,k^{\alpha,\beta}_{x_n}\right\}.$$ Then, $M_{\varphi}$ is an isometry on $H^2_{\alpha,\beta}$ and 
    \begin{equation}
        \operatorname{Im}(P_{\mathcal{M_{\alpha,\beta}}}(1 - M_{\varphi} M_{\varphi^*})|_{\mathcal{M_{\alpha,\beta}}})\subseteq P_{\mathcal{M_{\alpha,\beta}}}K_{\varphi}. \label{eq:8}
    \end{equation}
\end{lem}
\begin{proof}
     Since $M_{\varphi}$ is an isometry on $H^2$ and $H^2_{\alpha,\beta}$ is invariant, $M_{\varphi}$ is an isometry on $H^2_{\alpha,\beta}$. Moreover, clearly $K_{\varphi} = \operatorname{span}\left\{k^{\alpha,\beta}_{x_1},\ldots,k^{\alpha,\beta}_{x_n}\right\} \subset H^2_{\alpha,\beta} \ominus (\varphi H^2_{\alpha,\beta})$. On the other hand, if $g \in H^2_{\alpha,\beta} \ominus K_{\varphi}$, then $g$ vanishes on the zeroes of $\varphi$. In particular, there exists $h \in H^2$, such that $\varphi h = g$. By our assumption, there exists $c,d \in \C$, such that $d \neq 0$, $\varphi(0) = \varphi(\lambda) = d$, $g(0) = c \alpha$, and $g(\lambda) = c(\alpha + f_{\lambda}(\lambda)\beta)$. Hence, $h \in H^2_{\alpha,\beta}$. Conclude that $K_{\varphi} = H^2_{\alpha,\beta} \ominus (\varphi H^2_{\alpha,\beta})$. Let $P_{\varphi} = 1 - M_{\varphi} M_{\varphi}^* \in B(H^2_{\alpha,\beta})$ be the projection onto $K_{\varphi}$. Hence,
     \[
     \operatorname{Im}(P_{\mathcal{M_{\alpha,\beta}}}P_{\varphi}|_{\mathcal{M_{\alpha,\beta}}})\subseteq P_{\mathcal{M_{\alpha,\beta}}}K_{\varphi}.
     \]
\end{proof}

\begin{prop}
     For every pair $(\alpha,\beta)\in\mathcal{A}$, if
    $f\in \operatorname{span}\left\{k^{\alpha,\beta}_{z_\ell}\right\}\subset \mathcal{M_{\alpha,\beta}}$ satisfies
    \begin{equation} P_{\mathcal{M_{\alpha,\beta}}}(1-M_{\varphi}M_{\varphi}^*)|_{\mathcal{M_{\alpha,\beta}}}\geq ff^*,\label{eq:f.in.image}\end{equation}
        for every Blaschke product $\varphi \in \node$ vanishing on $z_{\tau(\ell)}$, then $f=0$.
        \label{prop.f=0}
    \end{prop}
    \begin{proof}
        Assume that
        $$\varphi(z)=\begin{cases}
            B_{z_3,z_4,\zeta}&\ell\in\left\{1,2\right\}
            \\
            B_{z_1,z_2,\omega}&\ell\in\left\{3,4\right\}
        \end{cases}$$
        with $\omega$ such that $B_{z_1,z_2,\omega} \in \node$, and $\zeta$ guaranteed from Lemma \ref{lem:existence.of.point}.
        By our assumptions, $\varphi\in\node$.
        Assume $\ell\in\left\{1,2\right\}$. In that case, $\tau(\ell)\in\{3,4\}$.
        By Lemma \ref{rmrk:0},
        $$Im(1-P_{\mathcal{M_{\alpha,\beta}}}M_{\varphi}M_{\varphi}^*|_{\mathcal{M_{\alpha,\beta}}})\subset \operatorname{span}\left\{k^{\alpha,\beta}_{z_3},\,k^{\alpha,\beta}_{z_4},\,P_{\mathcal{M_{\alpha,\beta}}}k^{\alpha,\beta}_{\zeta}\right\}.$$
        
        By $\eqref{eq:f.in.image}$, $f\in \operatorname{span}\left\{k^{\alpha,\beta}_{z_3},\,k^{\alpha,\beta}_{z_4},\,P_{\mathcal{M_{\alpha,\beta}}}k^{\alpha,\beta}_{\zeta}\right\}$. Hence, by the assumption on $f$, either $f=0$ or
        \[k^{\alpha,\beta}_{z_\ell}\in \operatorname{span}\left\{k^{\alpha,\beta}_{z_3},\,k^{\alpha,\beta}_{z_4},\,P_{\mathcal{M_{\alpha,\beta}}}k^{\alpha,\beta}_{\zeta}\right\},\]
      which contradicts (5) in Definition \ref{dfn:big}. Hence,  $f=0$ for all $(\alpha,\beta)\in\mathcal{A}$. The case of $\ell\in\{3,4\}$ is similar.
\end{proof}

\begin{lem}
    Let $\omega\in\mathbb{D}$, $\varphi$ analytic in $\mathbb{D}$ and define
    $$T=\begin{pmatrix}
    C&0
    \\
    (C-\omega I)v&\omega
    \end{pmatrix}$$
    for some matrix $C \in M_n$ with spectrum contained in $\mathbb{D}$ and a vector $v \in \C^n$. Then
    $$\varphi(T)=\begin{pmatrix}
        \varphi(C)&0
        \\
        v^*(\varphi(C)-\varphi(w)I)&\varphi(w)
    \end{pmatrix}.$$
    \label{lem.matrix.powers}
\end{lem}
\begin{proof}
    Set
    \[
    S = \begin{pmatrix} I_n & 0 \\ v^* & 1 \end{pmatrix} \in M_{n+1}.
    \]
    It is not hard to check that 
    \[
    S T S^{-1} = \begin{pmatrix} C & 0 \\ 0 & \omega \end{pmatrix}.
    \]
    Since both the spectrum of $C$ and $\omega$ lie in $\D$, the claim follows from functional calculus. 
\end{proof}
\begin{prop}

    Let $(\alpha,\beta)\in\mathcal{A}$ and let $\psi_{\alpha,\beta} \colon \node/I \to B(\cM_{\alpha,\beta} \oplus \C)$ be a coextension of $\rho_{\alpha,\beta}$. Write $$\psi_{\alpha,\beta}(h+I)^*=\begin{pmatrix}
    \rho_{\alpha,\beta}(h+I)^* &v\\0&z_\ell
\end{pmatrix}.$$ Then, $v=0$.\label{prop.no-non-trivial.coextensions}
\end{prop}
\begin{proof}
    By Lemma \ref{lem:diagonalizable}, $\psi_{\alpha,\beta}(h+I)$ is diagonalizable. Taking $p(z)=\prod_{i=1}^4(z-z_i)z(z-\lambda)$ and abbreviating $B_{\alpha,\beta}=\rho_{\alpha,\beta}(h + I)$, we have
    $$0=p(\psi_{\alpha,\beta}(h+I)^*)=\begin{pmatrix}
    0&\tilde{p}(B_{\alpha,\beta})v
    \\
    0&0
\end{pmatrix},$$
with $\tilde{p}(z)=\prod_{i\neq \ell}^4 (z-z_i)z(z-\lambda)$. Therefore, $v\in\ker\left(\tilde{p}(B_{\alpha,\beta}^*)\right)$. By Remark \ref{remark.minimal.polynomial}, the operator $B_{\alpha,\beta}^*$ is diagonalizable with distinct eigenvalues. Hence, there exists some $f\in \mathcal{M}_{\alpha,\beta}$, such that $v=(B_{\alpha,\beta}^*-z_\ell I)f$. Note that if $f=\sum_{i=1}^{4}\gamma_i k_{z_i}^{\alpha,\beta}$, then
$$\left(B_{\alpha,\beta}^*-z_\ell I\right)f=\left(B_{\alpha,\beta}^*-z_\ell  I\right)\underbrace{\left(\sum_{i\neq\ell}^4\gamma_i k^{\alpha,\beta}_{z_i}\right)}_g.$$
So we can assume $f\in \operatorname{span}\{k^{\alpha,\beta}_{z_i} \mid 1\leq i\neq\ell\leq 4\}$.
By Lemma \ref{lem.matrix.powers}, for every $\varphi\in \node$ analytic in a neighborhood of $\bb{D}$,
$$\varphi\left(\psi_{\alpha,\beta}(h+I)\right)^*=\begin{pmatrix}
\varphi(B_{\alpha,\beta})^*&(\varphi(B_{\alpha,\beta})^*-\overline{\varphi(z_\ell)})f\\0& \overline{\varphi(z_\ell)}
\end{pmatrix}.$$ In particular, if $\varphi(\overline{z_\ell})=\varphi(z_\ell)=0$ (as $z_\ell\in\bb{R}$), then,
$$\varphi(\psi_{\alpha,\beta}(h+I))^*=\begin{pmatrix}
\varphi(B_{\alpha,\beta})^*&\varphi(B_{\alpha,\beta})^*f\\0&0
\end{pmatrix}.$$ 
By Lemma \ref{lem:diagonalizable}, $\varphi\left(\psi_{\alpha,\beta}\left(h+I\right)\right)=\psi_{\alpha,\beta}\left(\varphi+I\right)$. Since $\|\varphi+I\|\leq 1$ ($\varphi$ is a Blaschke product) and $\psi_{\alpha,\beta}$ is an operator algebra representation, $\varphi\left(\psi_{\alpha,\beta}\left(h+I\right)\right)$ is a contraction. Thus,
$$1-\varphi\left(\psi_{\alpha,\beta}\left(h+I\right)\right)^*\varphi\left(\psi_{\alpha,\beta}\left(h+I\right)\right)\geq 0.$$
In particular,
\begin{equation}
    \begin{pmatrix}
        1-\varphi(B_{\alpha,\beta})^*(I+ff^*)\varphi(B_{\alpha,\beta})&0\\0&1
    \end{pmatrix}\geq 0.\label{equation1}
\end{equation}
Let $Q$ be the positive root of $I+ff^*$. It is immediate that $Q^{-2}=I-\frac{1}{1+\|f\|^2}ff^*$. Moreover, by $\eqref{equation1}$, $\varphi(B_{\alpha,\beta})^*Q$ is a contraction. Therefore,
$$Q^{-2}\geq \varphi(B_{\alpha,\beta})\varphi(B_{\alpha,\beta})^*\Rightarrow I-\varphi(B_{\alpha,\beta}) \varphi(B_{\alpha,\beta})^*\geq\frac{1}{1+\|f\|^2}ff^*.$$ This can be rewritten as, by definition of $B_{\alpha,\beta}^*$,
$$I-P_{\mathcal{M_{\alpha,\beta}}}M_{\varphi}M_{\varphi}^*|_{\mathcal{M_{\alpha,\beta}}}\geq\frac{1}{1+\|f\|^2}ff^*.$$

Letting 
$\varphi(z)=\frac{z-z_\ell}{1-\overline{z_\ell} z}\cdot\frac{z-z_{\tau(\ell)}}{1-\overline{z_{\tau(\ell)}}z}\in \node$,
by Lemma \ref{rmrk:0},
$$\sqrt{\frac{1}{1+\|f\|^2}}f\in \operatorname{span}\{k^{\alpha,\beta}_{z_\ell},\,k^{\alpha,\beta}_{z_{\tau(\ell)}}\}.$$ Hence, by our assumption, $f \in \C k_{z_{\tau(\ell)}}$. The second part of Proposition \ref{prop.f=0} implies $f=0$. Therefore, for every $(\alpha,\beta)\in\mathcal{A}$, $v=0$.
\end{proof}
Now, we can prove that $\rho_{\alpha,\beta}$ has no non-trivial coextensions.
\begin{proof}[Proof of Theorem \ref{thm.no.non.trivial.coextensions}]
    Fix $(\alpha,\beta)$ and assume that $\psi \colon \node/I\to B(\mathcal{M}_{\alpha,\beta}\oplus\mathcal{H})$ is a co-extension of $\rho_{\alpha,\beta}$. Write
    $$\psi(h+I)^*=\begin{pmatrix}
    \rho_{\alpha,\beta}(h+I)^*&D\\0&G
\end{pmatrix}.$$
Let $m(x)=\prod_{k=1}^4 (z-z_i)$. As $m(\psi^*(h+I))=\psi^*(m(h)+I)=0$, $m(G)=0$, that is $(G-z_1I)\cdot\ldots\cdot(G-z_4I)=0$. Therefore, there is a non-orthogonal direct sum decomposition
$$\mathcal{H}=\sum_{i=1}^4\ker(G-z_i I).$$
Assume that $D\neq 0$. Hence, there exists a vector $0\neq v\in\ker(G-z_i)$ for some $i$ such that $e=Dv\neq 0$. Consider $\mathcal{K}=\mathcal{M_{\alpha,\beta}}\oplus \operatorname{span}\{v\}\subset\mathcal{\mathcal{M_{\alpha,\beta}}}\oplus\mathcal{H}$. Then,
$$\psi(h+I)^*|_{\mathcal{K}}=\begin{pmatrix}
B_{\alpha,\beta}^*&e\\0&z_i
\end{pmatrix}.$$
By Proposition \ref{prop.no-non-trivial.coextensions}, $e=0$. Hence, $D=0$.
\end{proof}
\begin{thm}
    For every $(\alpha,\beta)\in\mathcal{A}$, $\rho_{\alpha,\beta}$ is extension maximal.\label{thm.no.non.trivial.extensions}
\end{thm}
We start with a proposition similar to Proposition \ref{prop.no-non-trivial.coextensions}.
\begin{prop}
    Let $(\alpha,\beta)\in\mathcal{A}$ and let $\psi_{\alpha,\beta} \colon \node/I \to B(\cM_{\alpha,\beta} \oplus \C)$ be an extension of $\rho_{\alpha,\beta}$. Write $$\psi_{\alpha,\beta}(h+I)=\begin{pmatrix}
    \rho_{\alpha,\beta}(h+I) &v\\0&z_\ell
\end{pmatrix}.$$ Then, $v=0$.\label{prop.no-non-trivial.extensions}
\end{prop}
\begin{proof}
    By the same reasoning as in the proof of Proposition \ref{prop.no-non-trivial.coextensions}, $v=(B_{\alpha,\beta}-z_\ell I)f$ for some $f\in \mathcal{M}_{\alpha,\beta}$.
    Write $f=\sum_{i=1}^4\gamma_i f_i$ with $f_i$ being the function satisfying $f_i(z_j)=\delta_{i,j}$, and observe that
    $$(B_{\alpha,\beta}-z_\ell I)f=(B_{\alpha,\beta}-z_\ell I)\left(\sum_{i=1}^4\gamma_i f_i\right)=(B_{\alpha,\beta}-z_\ell I)(\sum_{i\neq \ell}^4\gamma_i f_i).$$
    Since $f_\ell\in\ker(B_{\alpha,\beta}-z_\ell I)$. Therefore, we can assume that $f\in\operatorname{span}\{f_i \mid 1\leq i\neq \ell\leq 4\}$.
    Let $$g(z)=\frac{z-z_{\ell}}{1-z_\ell z}\cdot\frac{z-z_{\tau(\ell)}}{1-z_{\tau(\ell)}z}.$$ By Lemma \ref{lem:diagonalizable},
    $$\sigma_{\alpha,\beta}(g+I)=g(\sigma_{\alpha,\beta}(h+I))=\begin{pmatrix}
    g(B_{\alpha,\beta})&g(B_{\alpha,\beta})f\\0&0
\end{pmatrix}.$$
As $\sigma_{\alpha,\beta}$ is a representation and $g$ is a Blaschke product, hence a contraction,
$$1\geq\begin{pmatrix}
g(B_{\alpha,\beta})g(B_{\alpha,\beta})^*+g(B_{\alpha,\beta})ff^*g(B_{\alpha,\beta})^*&0\\0&1
\end{pmatrix}.$$
In particular,
$$1-P_{\mathcal{M_{\alpha,\beta}}}M_{g}M_{g}^*|_{\mathcal{M_{\alpha,\beta}}}\geq g(B_{\alpha,\beta})ff^*g(B_{\alpha,\beta})^*.$$
Hence, $$g(B_{\alpha,\beta})f=\sum_{i\neq\ell,\,\neq\tau(\ell)}^4\gamma_ig(z_i)f_i\in Im(1-P_{\mathcal{M_{\alpha,\beta}}}M_{g}M_{g}^*|_{\mathcal{M_{\alpha,\beta}}})\subset \operatorname{span}\{k^{\alpha,\beta}_{z_\ell},k^{\alpha,\beta}_{z_\tau(\ell)}\}.$$
Note that for $i\neq j$, $\ip{f_i}{k^{\alpha,\beta}_{z_j}}=0$. Therefore, $$\operatorname{span}\{f_i \mid 1\leq i\leq 4,\,i\neq \ell,\,\tau(\ell)\}\cap \operatorname{span}\{k^{\alpha,\beta}_{z_\ell},k^{\alpha,\beta}_{z_\tau(\ell)}\}=\{0\}.$$ Hence, $f\in \operatorname{span}\{f_{\tau(\ell)}\}$, and we can write $f=\gamma\cdot f_{\tau(\ell)}$. Now let
$$
\tilde{\varphi}(z)=\begin{cases}
        B_{z_1,z_2,\omega}&\ell\in\left\{1,2\right\},
            \\
            B_{z_3,z_4,\zeta}&\ell\in\left\{3,4\right\}.
        \end{cases}
$$
and assume again that $\ell\in\{1,2\}$.
By the same reasoning,
$$\tilde{\varphi}(A)f=\gamma\cdot \tilde{\varphi}(z_{\tau(\ell)})f_{\tau(\ell)}\in \operatorname{span}\{k^{\alpha,\beta}_{1},k^{\alpha,\beta}_{z_2}, P_{\cM_{\alpha,\beta}} k^{\alpha,\beta}_{\omega}\}.$$
By our choice of $\tilde{\varphi}$, $\tilde{\varphi}(z_{\tau(\ell)})\neq 0$.  Assume that $\gamma\neq 0$. Therefore, $\tilde{\varphi}(A)f\neq 0$ and consequently
$$f_{\tau(\ell)}\in \operatorname{span}\{k^{\alpha,\beta}_{z_1},k^{\alpha,\beta}_{z_2}, P_{\cM_{\alpha,\beta}} k^{\alpha,\beta}_{\omega}\}.$$ 
contradiction Definition \ref{dfn:big}. Hence, for all $(\alpha,\beta)\in\mathcal{A}$, $f=0$ and consequently $v=0$. The case where $\ell\in\{3,4\}$ is similar.
\end{proof}

\begin{proof}[Proof of Theorem \ref{thm.no.non.trivial.extensions}]
    Assume that $\psi \colon \node/I\to B(\mathcal{M}_{\alpha,\beta}\oplus\mathcal{H})$ is an extension of $\rho_{\alpha,\beta}$. Write 
    $$\psi(h+I) =\begin{pmatrix}
        \rho_{\alpha,\beta}(h+I)&D
        \\
        0&G
    \end{pmatrix}.$$
  Let $m(x)=\prod_{k=1}^4 (z-z_i)$. As $m(\psi(h+I))=\psi(m(h)+I)=0$, $m(G)=0$, that is $(G-z_1I)\cdot\ldots\cdot(G-z_4I)=0$. Therefore, we have a non-orthogonal direct sum decomposition
    $$\mathcal{H}=\sum_{i=1}^4\ker(G-z_iI).$$
    Assume that $D\neq 0$. Hence, there exists a vector $0\neq v\in\ker(G-z_iI)$ for some $i$ such that $e=Dv\neq 0$. Consider $\mathcal{K}=\mathcal{M_{\alpha,\beta}}\oplus \operatorname{span}\{v\}\subset\mathcal{\mathcal{M_{\alpha,\beta}}}\oplus\mathcal{H}$. By assumption, $\mathcal{K}$ is invariant. Therefore,
    $$\psi(h+I)|_{\mathcal{K}}=\begin{pmatrix}
    A&e\\0&z_i
\end{pmatrix}.$$
By Proposition \ref{prop.no-non-trivial.extensions}, $e=0$. Hence, $D=0$.
\end{proof}
For completeness we include the proof that for every $(\alpha,\beta)\in\mathcal{A}$, $\rho_{\alpha,\beta}$ is dilation maximal.
\begin{proof}[Proof of Theorem \ref{thm:good_then_big_envelope}]
    Let $(\alpha,\beta)\in\mathcal{A}$. Assume that
    $$\psi \colon \node/{I}\to B(\mathcal{K})$$ is a dilation of $\rho_{\alpha,\beta}$, i.e., $\cM_{\alpha,\beta} \subset \cK$ and 
    $$\rho_{\alpha,\beta}=P_{\mathcal{M_{\alpha,\beta}}}\psi|_{\mathcal{M_{\alpha,\beta}}}$$
    Then, again by Sarason's Lemma $$\psi(h+I)=\begin{pmatrix}
    A&B&C
    \\
    0&\rho_{\alpha,\beta}(h+I)&D
    \\
    0&0&E
\end{pmatrix}.$$
As the $2\times 2$ upper-left corner is a co-extension of $\rho_{\alpha,\beta}$, by Theorem \ref{thm.no.non.trivial.coextensions}, $B=0$. Similarly, since the $2\times 2$ lower-right corner is an extension of $\rho_{\alpha,\beta}$, by Theorem \ref{thm.no.non.trivial.extensions}, $D=0$. Therefore,
$$\psi(h+I)=\begin{pmatrix}
A&0&C\\0&\rho_{\alpha,\beta}&0\\0&0&E
\end{pmatrix}.$$
That is, $\mathcal{M_{\alpha,\beta}}$ is reducing for $\phi(h+I)$ and, thus, for all of $\psi(\node/I)$, as desired.
\end{proof}

Now it remains to show that there are sets of good points for some $0 \neq \lambda \in \D \cap \R$. We will dedicate the rest of this section to an example of a good set of points. Fix $\lambda=\frac{1}{\sqrt{2}}$. We will show that for the choice of points:
$$z_1=\frac{4}{3\sqrt{2}},\,z_2=\frac{1}{2\sqrt{2}},\,z_3=\frac{\sqrt{2}}{3},\,z_4=-\frac{1}{\sqrt{2}},$$
the algebra $C^*_e(\node/I)$ is infinite-dimensional. By what we observed, it is enough to show that these points are good with respect to $\lambda$. Our goal is to prove the following theorem.
\begin{thm} \label{thm:example_good}
    $z_1,\ldots,z_4$ are good points.
\end{thm}
Due to the fact that the proof is quite computational, we again break it into pieces.
\begin{prop}
    $z_1,\ldots,z_4$ satisfy conditions $(1)-(4)$.
\end{prop}
\begin{proof}
    With this setup
$$B_{\lambda}(z)=\frac{\sqrt{2}z\left(\sqrt{2}z-1\right)}{2-\sqrt{2}z},\,f_{\lambda}(z)=\frac{z}{\sqrt{2}-z}.$$
Recall that for a fixed $(\alpha,\beta)\in\G(1\times 1)$, the kernel function $k^{\alpha,\beta}$ is given by
$$k^{\alpha,\beta}(z,w)=\overline{\left(\alpha+\beta f_{\lambda}(w)\right)}(\alpha+\beta f_{\lambda}(z))+\frac{B_{\lambda}(z)\cdot \overline{B_{\lambda}(w)}}{1-z\overline{w}}.$$
To make calculations more manageable, we added two tables of values that will be useful to us. Here we use the assumptions $|\alpha|^2+|\beta|^2=1$ and $\alpha, \beta \in \R$.
\begin{table}[h!]
    \centering
    \begin{minipage}{0.45\textwidth}
        \centering
        \begin{tabular}{c|c|c}
            \hline
            $z$ & $B_{\lambda}(z)$ & $f_{\lambda}(z)$ \\ 
            \hline
            \\[-1.5ex]
            $z_1$ & $\frac{2}{3}$ & $2$ \\ 
            \\[-1.5ex]
            $z_2$ & $-\frac{1}{6}$ & $\frac{1}{3}$ \\ 
            \\[-1.5ex]
            $z_3$ & $-\frac{1}{6}$ & $\frac{1}{2}$ \\ 
            \\[-1.5ex]
            $z_4$ & $\frac{2}{3}$ & $-\frac{1}{3}$ \\ 
            \hline
        \end{tabular}
        \caption{Values of $B_{\lambda}(z)$ and $f_{\lambda}(z)$ for different $z$.}
        \label{tab:values_B}
    \end{minipage}
    \hspace{0.5cm} % Adjust the space between the tables
        \begin{minipage}{0.45\textwidth}
        \centering
        \begin{tabular}{c|c}
            \hline
            $(z,w)$ & $k^{\alpha,\beta}(z,w)$ \\ 
            \hline
            \\[-1.5ex]
            $(z_1, z_1)$ & $(\alpha + 2\beta)^2 + 4$ \\ 
            \\[-1.5ex]
            $(z_1, z_2)$ & $\frac{1}{3}\alpha^2 + \frac{7}{3}\alpha\beta + \frac{1}{2}$ \\ 
            \\[-1.5ex]
            $(z_1, z_3)$ & $\frac{5}{2}\alpha\beta + \frac{4}{5}$ \\ 
            \\[-1.5ex]
            $(z_1, z_4)$ & $\frac{5}{3}\alpha^2 + \frac{5}{3}\alpha\beta - \frac{2}{5}$ \\ 
            \\[-1.5ex]
            $(z_2, z_2)$ & $(\alpha + \frac{1}{3}\beta)^2 + \frac{2}{63}$ \\ 
            \\[-1.5ex]
            $(z_2, z_3)$ & $\frac{5}{6}\alpha^2 + \frac{5}{6}\alpha\beta + \frac{1}{5}$ \\ 
            \\[-1.5ex]
            $(z_2, z_4)$ & $\frac{10}{9}\alpha^2 - \frac{1}{5}$ \\ 
            \\[-1.5ex]
            $(z_3, z_3)$ & $(\alpha + \frac{1}{2}\beta)^2 + \frac{1}{28}$ \\ 
            \\[-1.5ex]
            $(z_3, z_4)$ & $\frac{7}{6}\alpha^2 + \frac{1}{6}\alpha\beta - \frac{1}{4}$ \\ 
            \\[-1.5ex]
            $(z_4,z_4)$ & $(\alpha - \frac{1}{3} \beta)^2 + \frac{8}{9}$ \\
            \hline
        \end{tabular}
        \caption{Values of $k^{\alpha,\beta}(z,w)$ for different pairs $(z,w)$.}
        \label{tab:values_k}
    \end{minipage}
\end{table}

In particular, we see that condition $(1)$ of Definition \ref{dfn:big} is satisfied.
Secondly, it is straightforward to check that
$$B_{z_1,z_2,\sqrt{2}-1},\,B_{z_3,z_4,\sqrt{2}-1}\in\node.$$ That gives us condition $(2)$.
%Now we check conditions $(3)$ and $(4)$. 
For condition $(3)$, consider the polynomial $f(\alpha,\beta) = \alpha^2 + \beta^2 - 1$. 
 Clearly, the real zeroes of $f$ are the circle. We will show that the first column/row is non-zero except for a finite subset of $(\alpha,\beta)$ in the circle. To see this, consider additionally $g_2(\alpha,\beta) = \alpha^2 + 7\alpha\beta + \frac{3}{2}$, $g_3(\alpha,\beta) = 5\alpha\beta + \frac{8}{5}$, and $g_4(\alpha,\beta) = 5\alpha^2 + 5\alpha\beta - \frac{6}{5}$. Since the vanishing of $g_i$ is equivalent to the vanishing of $k^{\alpha,\beta}(z_1,z_i)$, we get that the set of points $(\alpha,\beta)$, where the first column has zeroes in it, is contained in
\[
D = \bigcup_{i=2}^4 Z(f) \cap Z(g_i).
\]
Here, $Z(f)$ stands for the zero locus of $f$ and similarly for the $g_i$. Since these are algebraic planar curves each of degree $2$, that do not have components in common, each intersection set has at most $4$ real points. Therefore, $D$ is finite. Therefore, on the complement of $D$, condition $(3)$ of Definition \ref{dfn:big} holds.

As for condition $(4)$, we first note that if we multiply $(\alpha,\beta)$ by a unimodular scalar, we do not change the kernel. Hence, we are considering, in fact, the real projective line. However, we will restrict ourselves to a half-circle. Let $A'\subset\bb{T}$ be the upper open half-circle. We will show that for every $(\alpha,\beta)\neq (\alpha',\beta') \in \mathcal{A}'$, either \eqref{eq:0} is satisfied for $(z_1,z_2)$ or for $(z_2,z_3)$.
To this end, fix $(\alpha,\beta)\neq(\alpha',\beta')\in \mathcal{A}'$ and let $\frac{\beta}{\alpha}=x$. Since $\alpha^2+\beta^2=1$,
    $$x^2=\left(\frac{\beta}{\alpha}\right)^2=\frac{1-\alpha^2}{\alpha^2}\Rightarrow \alpha^2=\frac{1}{1+x^2}.$$
    We can rewrite the kernels in terms of $x$
    \[k^{\alpha,\beta}(z_1,z_1)=(\alpha+2\beta)^2+4=\alpha^2\left(1+2\frac{\beta}{\alpha}\right)\left(1+2\frac{\beta}{\alpha}\right)+4=\]
    \[=\frac{1}{1+x^2}(1+2x)^2+4=\frac{8x^2+4x+5}{1+x^2}.\]
    Similarly,
    $$k^{\alpha,\beta}(z_2,z_2)=\frac{1}{1+x^2}\left(1+\frac{x}{3}\right)^2+\frac{2}{63}=\frac{9x^2+42x+65}{63(1+x^2)},$$
    
    $$k^{\alpha,\beta}(z_3,z_3)=\frac{1}{1+x^2}(1+\frac{x}{2})^2+\frac{1}{28}=\frac{8x^2+28x+29}{28(1+x^2)},$$
    $$k^{\alpha,\beta}(z_1,z_2)=\frac{1}{1+x^2}\left(1+\frac{x}{3}\right)(1+2x)-\frac{1}{6}=\frac{3x^2+14x+5}{6(1+x^2)},$$
    
    $$k^{\alpha,\beta}(z_2,z_3)=\frac{1}{1+x^2}\left(1+\frac{x}{2}\right)\left(1+\frac{x}{3}\right)+\frac{1}{30}=\frac{6x^2+25x+31}{30(1+x^2)}.$$
    By letting $y=\frac{\beta'}{\alpha'}$ and replacing $x$ with $y$ in the calculations above, we can write $k^{\alpha',\beta'}(z_i,z_j)$ in terms of $y$. By taking $j=1,\,i=2$, Equation \eqref{eq:0} becomes
    \begin{equation}
        \frac{\left(\frac{\left(3x^{2}+14x+5\right)}{6\left(1+x^{2}\right)}\right)^{2}}{\frac{\left(8x^{2}+4x+5\right)}{1+x^{2}}\cdot\frac{\left(9x^{2}+42x+65\right)}{1+x^{2}}}=\frac{\left(\frac{\left(3y^{2}+14y+5\right)}{6\left(1+y^{2}\right)}\right)^{2}}{\frac{\left(8y^{2}+4y+5\right)}{1+y^{2}}\cdot\frac{\left(9y^{2}+42y+65\right)}{1+y^{2}}}.
    \end{equation}
    Simplifying and canceling some terms on both sides, we obtain the equation
    \[
    (3x^2+14x+5)^2(8y^2+4y+5)(9y^2+42y+65)-(3y^2+14y+5)^2(8x^2+4x+5)(9x^2+42x+65)=0.
    \]
    As we are looking for solutions with $x\neq y$, and as $x=y$ is a solution, by dividing by $x-y$, refactoring, and dividing both sides by the constants, the equations become
    \begin{equation}
        \left(xy+x+y+3\right)\left(12x^{2}y^{2}+31x^{2}y+31xy^{2}-5x^{2}+28xy-5y^{2}-65x-65y-50\right)=0\label{eq:5}
    \end{equation}
    We do the same analysis for $j=2,\,i=3$. For this choice of indices, \eqref{eq:0} becomes
    \begin{equation}\frac{\left(\frac{\left(6x^{2}+25x+31\right)}{30\left(1+x^{2}\right)}\right)^{2}}{\frac{\left(9x^{2}+42x+65\right)}{1+x^{2}}\cdot\frac{\left(8x^{2}+28x+29\right)}{63\left(1+x^{2}\right)}}=\frac{\left(\frac{\left(6y^{2}+25y+31\right)}{30\left(1+y^{2}\right)}\right)^{2}}{\frac{\left(9y^{2}+42y+65\right)}{1+y^{2}}\cdot\frac{\left(8y^{2}+28y+29\right)}{63\left(1+y^{2}\right)}}\label{eq:1}\end{equation}
        Again, canceling some terms on both sides, dividing by $x-y$, refactoring, and dividing both sides by the relevant constants, $\eqref{eq:1}$ becomes
        $$\left(xy+x+y-1\right)\left(12x^{2}y^{2}+25x^{2}y+25xy^{2}+37x^{2}+37y^{2}+25x+25y+62\right)=0.$$
        By our analysis, $(x,y)$ with $x\neq y$ solves $\eqref{eq:5}$ if and only if
        \begin{equation}
            xy+x+y+3=0.\label{eq:2}
        \end{equation} or
        \begin{equation} 12x^{2}y^{2}+31x^{2}y+31xy^{2}-5x^{2}+28xy-5y^{2}-65x-65y-50=0.\label{eq:3}\end{equation}
            and it solves $\eqref{eq:1}$ if and only if
            \begin{equation}
                xy+x+y-1=0.\label{eq:6}
            \end{equation}
            or
            \begin{equation}
                12x^{2}y^{2}+25x^{2}y+25xy^{2}+37x^{2}+37y^{2}+25x+25y+62=0.\label{eq:4}
            \end{equation}
            Clearly, $\eqref{eq:2}$ and $\eqref{eq:6}$ can not be satisfied simultaneously. Moreover, notice that
            $$12x^{2}y^{2}+25x^{2}y+25xy^{2}+37x^{2}+37y^{2}+25x+25y+62=$$$$=6x^2\left(y^2+\frac{25}{6}y+\frac{30.75}{6}\right)+6y^2\left(x^2+\frac{25}{6}x+\frac{30.75}{6}\right)+(2.5x+5)^2+(2.5y+5)^2+12=$$$$=6x^2\left(\left(y+\frac{25}{12}\right)^2+\frac{113}{144}\right)+6y^2\left(\left(x+\frac{25}{12}\right)^2+\frac{113}{144}\right)+(2.5x+5)^2+(2.5y+5)^2+12>0,$$ i.e., $\eqref{eq:4}$ is never satisfied. Therefore, it is enough to check that $\eqref{eq:6}$ and $\eqref{eq:3}$ are not satisfied simultaneously. Assume that $(x,y)$ with $x\neq y$ is a solution of $\eqref{eq:6}$. Plugging $xy=1-x-y$ into $\eqref{eq:3}$, we see that
            $$12(1-x-y)^2+31x(1-x-y)+31y(1-x-y)-5x^2+28(1-x-y)-5y^2-65x-65y-50=$$
            $$=-24x^2-38xy-24y^2-86x-86y-10=-24x^2-38(1-x-y)-24y^2-86x-86y-10=$$$$=-24(x^2+y^2+2x+2y+2)=-24((x+1)^2+(y+1)^2)<0,$$
            i.e., $(x,y)$ is not a solution of $\eqref{eq:3}$. Therefore, $\eqref{eq:0}$ can not be satisfied for both $(z_1,z_2)$ and $(z_2,z_3)$, hence condition $(4)$ is satisfied.
\end{proof}
It is left to prove our points satisfy condition $(5)$ of Definition \ref{dfn:big}. To this end, we first prove two technical lemmas.
\begin{lem}
            Let $\left\{z_1,\ldots,z_n\right\},\,\left\{w_1,\ldots,w_n\right\}\subset\bb{D}\setminus\left\{0,\lambda\right\}$ be two sets of points such that $z_i\neq z_j$ and $w_i\neq w_j$ for all $1\leq i\neq j\leq n$. Then,
            \[P=\Bigg[\frac{B_\lambda(z_i)\overline{B_{\lambda}(w_j)}}{1-z_i\overline{w_j}}\Bigg]_{i,j=1}^n\] is invertible.\label{lem.singularity.of.semi.pick.matrix}
        \end{lem}
        \begin{proof}
            As $B_{\lambda}(z_i)\neq 0$ and $B_{\lambda}(w_j)\neq 0$ by assumption and 
            $$P = \begin{pmatrix}
            B_{\lambda}(z_1)&0&\ldots&0
            \\
            0&B_{\lambda}(z_2)&\ldots&0
            \\
            \vdots&\vdots&\ddots&\vdots
            \\
            0&0&\ldots&B_{\lambda}(z_n)
        \end{pmatrix}\underbrace{\Big[\frac{1}{1-z_i\overline{w_j}}\Big]}_{\tilde{P}}\begin{pmatrix}
        \overline{B_{\lambda}(w_1)}&0&\ldots&0
        \\
        0&\overline{B_{\lambda}(w_2)}&\ldots&0
        \\
        \vdots&\vdots&\ddots&\vdots
        \\
        0&0&\ldots&\overline{B_{\lambda}(w_n)}
    \end{pmatrix},$$
    it is enough to prove $\tilde{P}$ is invertible. We prove this by induction on $n$. For $n=1$, the claim is clear. Assume that this is true for $n$, and consider the sets $\left\{z_1,\ldots,z_{n+1}\right\},\,\left\{w_1,\ldots,w_{n+1}\right\}$ such that they satisfy the assumption.
    Calculating the determinant of $\tilde{P}$ as a function of $z_1$ gives us
    $$\det(\tilde{P})=\sum_{j=1}^{n+1}(-1)^{1+j}\det((\tilde{P})_{1,j})\cdot\frac{1}{1-z_1\overline{w_j}},$$
    with $(\tilde{P})_{1,j}$ denoting the $(1,j)$'th minor of $\tilde{P}$. Replacing $z_1$ with a variable $z$ and thinking of $\det(\tilde{P})$ as a function of $z$, $\det(\tilde{P})(z)$ has poles at $\frac{1}{w_j}$ for all $1\leq j\leq n+1$ (by induction hypothesis, $\det((P)_{1,j})\neq 0$), and
    $$\det(\tilde{P})(z_2)=\ldots=\det(P)(z_{n+1})=\det(\tilde{P})(\infty)=0.$$
    i.e., it has $n+1$ zeroes at $z_2,\ldots,z_n,\infty$. Therefore, $\det(\tilde{P})=\det(\tilde{P})(z_1)\neq 0$.
    
\end{proof}
\begin{lem}
    Let $\left\{z_1,\ldots,z_n\right\},\,\left\{w_1,\ldots,w_n\right\}\subset(\bb{D}\cap\bb{R})\setminus\left\{0,\lambda\right\}$ be two sets of points such that $z_i\neq z_j$ and $w_i\neq w_j$ for all $1\leq i\neq j\leq n$. Then, the set of all pairs $(\alpha, \beta)\in \bb{T}$ for which the matrix
    $\left[k^{\alpha,\beta}(z_i,w_j)\right]_{i,j=1}^4$ is singular is either $\bb{T}$ or finite.
    \label{lem:algebraic.det.condition}
\end{lem}
\begin{proof}
    Rewrite the matrix as
    $$\underbrace{\begin{pmatrix}
    \alpha+f_{\lambda}(z_1)\beta
    \\
    \vdots
    \\
    \alpha+f_{\lambda}(z_n)\beta
\end{pmatrix}}_{u_{\alpha,\beta}}\cdot \underbrace{\begin{pmatrix}
\alpha+f_{\lambda}(w_1)\beta
\\
\vdots
\\
\alpha+f_{\lambda}(w_n)\beta
\end{pmatrix}^*}_{v^*_{\alpha,\beta}}+\underbrace{\Bigg[\frac{B_\lambda(z_i)\overline{B_{\lambda}(w_j)}}{1-z_i\overline{w_j}}\Bigg]}_{P}.$$
This matrix is a rank-one perturbation of the invertible matrix $P$. By the determinant lemma, its determinant is given by
$\det(P)(1+v^*_{\alpha,\beta}P^{-1}u_{\alpha,\beta})$.
As everything is real-valued, the polynomial $1+v^*_{\alpha,\beta}P^{-1}u_{\alpha,\beta}$ is a real polynomial of degree $2$ in $(\alpha,\beta)$. Note that the polynomial $1-\alpha^2 - \beta^2$ cuts out the unit circle. A straightforward computation shows that this polynomial is irreducible. The zero sets of two polynomials of degree $2$ can have either a finite intersection (including the empty set) or have a component in common. The circle is the set of real points in the complex variety cut out by the polynomial $1-\alpha^2-\beta^2$. Since it is infinite, it is Zariski dense in the curve. Therefore, if a polynomial of degree $2$ vanishes on an infinite set in the circle, it must vanish on the entire complex variety. We conclude that  the level-set
$$\left\{(\alpha,\beta)\in\bb{T} \mid 1+v^*_{\alpha,\beta}P^{-1}u_{\alpha,\beta}=0\right\}$$
is either finite or is equal to $\bb{T}$.
\end{proof}
\begin{prop}
    The points $z_1,\ldots,z_4$ satisfy condition $(5)$ of Definition \ref{dfn:big}.
\end{prop}
\begin{proof}
    First of all, we note that for our choice of $z_1,\ldots,z_4$, we can take
    $\omega=\zeta=\sqrt{2}-1$. Let $\sigma=(12)(34),\,\tau=(14)(23)\in S_4$.
    Assume that $\ell\in\{1,2\}$ and $(\alpha,\beta)\in\mathcal{A}'$ satisfies
    $$k_{z_{\ell}}\in \operatorname{span}\{k^{\alpha,\beta}_{z_{\tau(\ell)}},k^{\alpha,\beta}_{z_{\tau(\sigma(\ell))}},P_{M_{\alpha,\beta}}k_{\sqrt{2}-1}^{\alpha,\beta}\}.$$
    Since $k^{\alpha,\beta}_{z_{\tau(\ell)}},\,k^{\alpha,\beta}_{z_{\sigma(\tau(\ell))}},\,k^{\alpha,\beta}_{z_{\ell}}$ are linearly independent, there are scalars $\gamma_1,\gamma_2,\gamma_3$, such that
    $$P_{M_{\alpha,\beta}}k_{\sqrt{2}-1}^{\alpha,\beta}=\gamma_1k^{\alpha,\beta}_{z_{\ell}}+\gamma_2k^{\alpha,\beta}_{z_{\tau(\ell)}}+\gamma_3k^{\alpha,\beta}_{z_{\tau(\sigma(\ell))}}.$$
    By taking inner products with $k^{\alpha,\beta}_{z_i}$ on both sides for $1\leq i\leq 4$, 
    we see that the matrix
    $$C_{\ell}(\alpha,\beta)=\left[k^{\alpha,\beta}(z_i,w_j)\right]$$
    is singular, where
    $$w_j=\begin{cases}
        z_j&j\neq\sigma(\ell)
        \\
        \sqrt{2}-1&j=\sigma(\ell).
    \end{cases}$$
    Now assume that $1\leq\ell\leq 4$ and $(\alpha,\beta)\in\mathcal{A}$ satisfy
    $$f_{\ell} \in\operatorname{span}\{k^{\alpha,\beta}_{z_{\tau(\ell)}},k^{\alpha,\beta}_{z_{\tau(\sigma(\ell))}},P_{\mathcal{M}_{\alpha,\beta}}k^{\alpha,\beta}_{\sqrt{2}-1}\}.$$
    Then, there are scalars $\gamma_1,\,\gamma_2,\,\gamma_3$, such that
    $$f_{\ell}=\gamma_1k^{\alpha,\beta}_{z_{\tau(\ell)}}+\gamma_2k^{\alpha,\beta}_{z_{\tau(\sigma(\ell))}}+\gamma_3P_{\mathcal{M}_{\alpha,\beta}}k^{\alpha,\beta}_{\sqrt{2}-1}.$$
    By taking inner products with $k^{\alpha,\beta}_{z_i}$ for $1\leq i\neq\ell\leq 4$, this implies that the $3\times 3$ minor 
    $\left(C_{\ell}(\alpha,\beta)\right)_{\ell,\ell}$
    is singular.
    By taking $\alpha=1,\,\beta=0$ and calculating $k^{1,0}(z_\ell,\sqrt{2}-1)$ for all $1\leq\ell\leq 4$, we obtain the following four matrices:
    $$C_{1}(1,0)=\begin{pmatrix}
        5&\frac{17-8\sqrt{2}}{7}&\frac{4}{5}&\frac{19}{15}
        \\
        \frac{5}{6}&\frac{13-7\sqrt{2}}{3}&\frac{31}{30}&\frac{41}{45}
        \\
        \frac{4}{5}&\frac{5(\sqrt{2}-1)}{\sqrt{2}}&\frac{29}{28}&\frac{11}{12}
        \\
        \frac{19}{15}&\frac{5(2\sqrt{2}+1)}{21}&\frac{11}{12}&\frac{7}{6}
    \end{pmatrix},\,C_{2}(1,0)=\begin{pmatrix}
        \frac{17-8\sqrt{2}}{7}&\frac{5}{6}&\frac{4}{5}&\frac{19}{15}
        \\
        \frac{13-7\sqrt{2}}{3}&\frac{65}{63}&\frac{31}{30}&\frac{41}{45}
        \\
        \frac{5(\sqrt{2}-1)}{\sqrt{2}}&\frac{31}{30}&\frac{29}{28}&\frac{11}{12}
        \\
        \frac{5(2\sqrt{2}+1)}{21}&\frac{41}{45}&\frac{11}{12}&\frac{7}{6}
    \end{pmatrix},$$
    $$C_{3}(1,0)=\begin{pmatrix}
        
        5&\frac{5}{6}&\frac{4}{5}&\frac{17-8\sqrt{2}}{7}
        \\
        \frac{5}{6}&\frac{65}{63}&\frac{31}{30}&\frac{13-7\sqrt{2}}{3}
        \\
        \frac{4}{5}&\frac{31}{30}&\frac{29}{28}&\frac{5(\sqrt{2}-1)}{2}
        \\
        \frac{19}{15}&\frac{41}{45}&\frac{11}{12}
        &\frac{5(2\sqrt{2}+1)}{21}
    \end{pmatrix},\,C_4(1,0)=\begin{pmatrix}
        5&\frac{5}{6}&\frac{17-8\sqrt{2}}{7}&\frac{19}{15}\\\frac{5}{6}&\frac{65}{63}&\frac{13-7\sqrt{2}}{3}&\frac{41}{45}\\\frac{4}{5}&\frac{31}{30}&\frac{5(\sqrt{2}-1)}{2}&\frac{11}{12}\\\frac{19}{15}&\frac{41}{45}&\frac{5(2\sqrt{2}+1)}{21}&\frac{7}{6}
    \end{pmatrix}.$$
    The matrices $C_{\ell}(1,0)$ and $\left(C_{\ell}(1,0)\right)_{\ell,\ell}$ are invertible for all $1\leq\ell\leq 4$. By Lemmas \ref{lem.singularity.of.semi.pick.matrix} and \ref{lem:algebraic.det.condition}, this implies that for every $1\leq\ell\leq 4$, $C_{\ell}(\alpha,\beta)$ and $\left(C_{\ell}(\alpha,\beta)\right)_{\ell,\ell}$ are invertible for all $(\alpha,\beta)\in\mathcal{A}'\setminus S$ for some finite set $S\subset A'$. In particular, by taking $A=A'\setminus S$, $A$ is the desired set, and condition $(5)$ is satisfied.
\end{proof}

\begin{proof}[Proof of Theorem \ref{thm:example_good}]
    As demonstrated above, if we take $A$ as the complement of a finite set in the open upper-half circle, then the conditions of Definition \ref{dfn:big} are satisfied. Hence, our points are good.
\end{proof}

\begin{remark}
    A comment on the choices made in the example is in order. The technicalities in proving the unitary inequivalence forced us to restrict our attention to real $(\alpha,\beta)$ and real points. We believe that every $[\alpha:\beta] \in \mathbb{P}^1(\C)$ outside of a finite set, perhaps, gives rise to a boundary representation of $\node/I$. Therefore, $C(P^1(\C))$ is a quotient of $C^*_e(\node/I)$. It is also challenging to find a higher-dimensional boundary representation. It might be the case that most $(\alpha, \beta)$ in the Grassmannian will give such representations.
\end{remark}

\section{A $C^*$-cover candidate}\label{sec:cover}

The non-commutative Grassmanian, $G_n^{nc}$, introduced by Brown \cite{brown1981ext}, is the universal $C^*$-algebra generated by the identity element and $\{P_{i,j} \mid 1\leq i,j\leq n\}$, such that the matrix
$$P=[P_{i,j}]_{i,j=1}^n$$ is a projection. In \cite{mcclanahan1992c}, McClanahan proved that $G_2^{nc}$ is a $C^*$-subalgebra of the free product $M_2*_{\C}{\C}^2$. Consequently, by a theorem of Exel and Loring \cite[Theorem 3.2]{exel1992finite}, $G_2^{nc}$ is RFD. In this section, we present another characterization of $G_2^{nc}$, and define a completely isometric embedding of $H^{\infty}_{\text{node}}/I$ into $M_n(G_2^{nc})$, providing a candidate for the $C^*-$ envelope of $H^{\infty}_{\text{node}}/I$.
\\\\
    Let $\frak{A}$ be the universal $C^*$-algebra generated by two elements $x,y$ such that $x^*x+y^*y=1$.
    Note that such universal $C^*$-algebra exists by \cite[Section $2$, 8.3.1]{blackadar2006operator}, since the relation is of the form
$$\|p(x,y,x^*,y^*)\|= 0$$
for the non-commutative polynomial $p(x,y,x^*,y^*)=x^*x+y^*y-1$, this relation is realizable by bounded operators on a Hilbert space, and it places an upper bound on $\|x\|$ and $\|y\|$. This algebra is also known as the universal Pythagorean algebra (see, for example, \cite{BrothierJones} and \cite{CourtneySherman}). Moreover, it was shown in \cite[Theorem 6.7]{CourtneySherman} that $\mathfrak{A}$ is RFD.
\begin{thm}
    Let
    $\frak{B}=C^*(1,xx^*, yy^*, xy^*)\subset\frak{A}$.
    Then, $G_2^{nc}\cong\frak{B}$. \label{thm: characterization.of.NC.Grass}
\end{thm}

\begin{proof}
    As $\begin{pmatrix}
            x\\y
        \end{pmatrix}$ is a column isometry, the matrix
        $$\begin{pmatrix}
            x\\y
        \end{pmatrix}\begin{pmatrix}
            x\\y
        \end{pmatrix}^*=\begin{pmatrix}
            xx^*&xy^*\\yx^*&yy^*
        \end{pmatrix}$$ is a projection. By universality, if $a,\,b,\,c$ are the generators of $G_2^{nc}$,
        there exists a unital surjective $*$-homomorphism
        $$\rho \colon G_2^{nc}\to\frak{B},\, a\to xx^*,\,b\to xy^*,\,c\to yy^*.$$
        It is enough to prove that $\rho$ is injective. To this end, let $x\in\ker(\rho)$. As $G_2^{nc}$ is RFD, to show $x=0$ it is enough to prove that for every finite-dimensional representation $\pi$ of $G_2^{nc}$, $\pi(x)=0$.
        Fix a finite-dimensional representation
        $$\pi \colon G_2^{nc}\to M_n(\C).$$
        Then, the matrix
        $$P=\begin{pmatrix}
            \pi(a)&\pi(b)
            \\
            \pi(b)^*&\pi(c)
        \end{pmatrix}$$
        is a projection. Let $m=tr(P)$. We have two cases.
        \begin{enumerate}
            \item If $m\leq n$, let $\chi_1 \colon G_2^{nc}\to\C$ be the character satisfying
            $$\chi_{1}(a)=\chi_{1}(c)=1,\,\chi_1(b)=0$$
            and consider the representation $\pi'=\pi\oplus\chi_1^{\oplus (n-m)}$, with $\chi_1^{\oplus (n-m)}$ being the $(n-m)$-fold direct sum of $\chi_1$ with itself. Then,
            $$P'=\begin{pmatrix}
                \pi'(a)&\pi'(b)\\\pi'(b)^*&\pi'(c)
            \end{pmatrix}$$ is a projection in $M_{2(2n-m)}(\C)$ with $tr(P')=2(n-m)+m=2n-m$, hence there exists a column isometry
            $$V \colon \C^{2n-m}\to\C^{2(2n-m)},\,V=\begin{pmatrix}
                \tilde{x}\\\tilde{y}
            \end{pmatrix},$$
            such that $VV^*=P'$. By universality of $\frak{A}$, we have a representation
            $$\tilde{\pi} \colon \frak{A}\to M_{2n-m}(\C),\,\tilde{\pi}(x)=\tilde{x},\,\tilde{\pi}(y)=\tilde{y}.$$
            Therefore, $\pi'=\tilde{\pi}|_{\frak{B}}\circ\rho$. As $x\in\ker(\rho)$, $\pi'(x)=0$ and consequently $\pi(x)=0$.         
        
        \item If $n<m$, let $\chi_0 \colon G^{nc}_2\to\C$ be the character that maps $a$, $b$, and $c$ to zero, and consider the representation $\pi'=\pi\oplus\chi_0^{\oplus (m-n)}$.
        Then,
        $$P'=\begin{pmatrix}
                \pi'(a)&\pi'(b)\\\pi'(b)^*&\pi'(c)
            \end{pmatrix}$$ is a projection in $M_{2m}(\C)$ with $tr(P')=m$. By similar arguments, again $\pi(x)=0$.
         \end{enumerate}
         Since this holds for an arbitrary finite-dimensional representation $\pi$, we conclude that $x=0$. Hence, $\rho$ is an isomorphism of $C^*$-algebras.
\end{proof}

For every $m$ and every $(\alpha,\beta)\in\G(m\times m)$, $\begin{pmatrix}
    \alpha^t 
    \\
    \beta^t
\end{pmatrix}$ is a column isometry. By universality of $\frak{A}$, there exists a representation
$$\pi_{\alpha,\beta}\colon\frak{A}\to M_m(\C)$$
given by $x\to\alpha^t,\,y\to\beta^t$. Moreover, note that every $m$ dimensional representation of $\frak{A}$ is of this form for some pair $(\alpha,\beta)\in\G(m\times m)$. Define $\Psi\in M_n(\frak{A})$ by
$$\Psi=\sum_{i,j=1}^n\underbrace{\left(\left(x+f_{\lambda}(z_i)y\right)\left(x^*+\overline{f_{\lambda}(z_j)}y^*\right)+\frac{B_{\lambda}(z_i)\overline{B_{\lambda}(z_j)}}{1-z_i\overline{z_j}}\right)}_{l_{i,j}}F_{i,j},$$
where $F_{i,j}$ are the matrix units of $M_n(\C)$, $f_{\lambda}(z)=\frac{\sqrt{1-|\lambda|^2}}{|\lambda|}(k_{\lambda}(z)-1)$ and $B_{\lambda}$ is the Blaschke product with simple zeroes at $0$ and $\lambda$. An important observation is the following.
\begin{prop}
    Assuming $0,\lambda\not\in\{z_1,\ldots,z_n\}$, $\Psi$ is positive definite.\label{prop:positivity.of.psi}
\end{prop}
\begin{proof}
    Let $L=\begin{pmatrix}
    x+f_{\lambda}(z_1)y\\\vdots\\x+f_{\lambda}(z_n)y
\end{pmatrix}$. Then,
$$\Psi=LL^*+\underbrace{\left[\frac{B_{\lambda}(z_i)\overline{B_{\lambda}(z_j)}}{1-z_i\overline{z_j}}\right]_{i,j=1}^n}_{G}.$$
Note that $G$ is the Gram matrix of a linearly independent set, since the kernel functions of $H^2$ are linearly independent, our points do not contain the zeroes of $B_{\lambda}$, and $B_{\lambda}$ is an isometry. Hence, $G$ is strictly positive and so is $\Psi$.

\end{proof}
We are ready to present the main theorem of this section.
\begin{thm}
    The homomorphism
    $\Gamma\colon H^{\infty}_{\text{node}}/I\longmapsto M_n(\frak{A})$
    given by
    $$\Gamma(f+I)=\Psi^{-\frac{1}{2}}\underbrace{\begin{pmatrix}
        f(z_1)&0&\ldots&0
        \\
        0&f(z_2)&\ldots&0
        \\
        \vdots &\vdots&\ddots&\vdots
        \\
        0&0&\ldots&f(z_n)
    \end{pmatrix}}_{D_f}\Psi^{\frac{1}{2}}$$
    is completely isometric. In fact, $\Gamma(\node/I) \subset M_n(\fB)$. \label{thm:completely.isometric.embedding}
\end{thm}
To prove Theorem \ref{thm:completely.isometric.embedding}, we begin by reformulating the generalized form of the trace condition presented in Section \ref{sec:pick}.
For fixed $m,r\in\N$, $(\alpha,\beta)\in\G(m\times m)$ and $W_1,\ldots,W_n\in M_r(\C)$, define the generalized trace condition to be
\begin{equation}
    \sum_{i,j=1}^n \operatorname{tr}[X_jk^{\alpha,\beta}(z_i,z_j)X_i^*-W_j^*X_jk^{\alpha,\beta}(z_i,z_j)X_i^*W_i]\geq 0,\label{eq:gen.trace.con}
\end{equation}
for all $X_1,\ldots,X_n\in M_{r\times m}(\C)$.
Let $m,r\in\N$, and consider the space $M_{r, mn}(\C)$, thought of as the space of row block matrices with $n$ blocks of size $r\times m$. In other words, we identify $M_{r, mn} \cong M_{r,m} \otimes \C^n$. Fix some $W=(W_1,\ldots,W_n)\in M_r(\C)^{\oplus n}$ and define two linear operators
$$\psi_{\alpha,\beta},\,\varphi_W \colon M_{r\times mn}(\C)\to M_{r\times mn}(\C)$$
by 
$$\psi_{\alpha,\beta}\begin{pmatrix}
    X_1&\ldots&X_n
\end{pmatrix}=\begin{pmatrix}
    X_1&\ldots&X_n
\end{pmatrix}\cdot \left[k^{\alpha,\beta}(z_j,z_i)^t\right]_{i,j=1}^n$$
and 
$$\varphi_W\begin{pmatrix}
X_1&\dots&X_n
\end{pmatrix}=\begin{pmatrix}
W_1X_1&\dots&W_nX_n
\end{pmatrix}$$
Furthermore, define an inner product on $M_{r\times mn}(\C)$ by
$$\ip{\begin{pmatrix}
X_1\\\vdots\\X_n
\end{pmatrix}^T}{\begin{pmatrix}
Y_1\\\vdots\\Y_n
\end{pmatrix}^T}=\operatorname{tr}\left(\sum_{i=1}^n X_iY_i^*\right).$$
It is straightforward that with respect to this inner product
$$\varphi_W^*\begin{pmatrix}
    X_1&\ldots &X_n
\end{pmatrix}=\begin{pmatrix}
    W_1^*X_1&\ldots& W_n^*X_n
\end{pmatrix}$$

\begin{lem}
    Let $E_{i,j}$ be the matrix units in $M_{r\times mn}$, and define the ordered basis
    $$\varepsilon = \left\{E_{i,j} \mid 1\leq i\leq r,\, 1\leq j\leq mn\right\}$$
    ordered by $E_{i,j}<E_{k,l}$ if and only if $j<l$ or $j=l$ and $i<k$. Then, in Kronecker's notation \cite{horn1994topics}, for every $n\in\N$ and every $W=(W_1,\ldots,W_n)\in M_r(\C)^{\oplus n}$,
    $$[\psi_{\alpha,\beta}]_{\varepsilon}=\sum_{i,j=1}^n F_{j,i}\otimes k^{\alpha,\beta}(z_j,z_i)^t\otimes I_r,\,[\varphi_W]_{\varepsilon}=\sum_{i=1}^n F_{i,i}\otimes I_m\otimes W_i$$\label{lem:rep.matrix.of.operators.} with $F_{i,j}$ the matrix units of $M_n(\C)$.
\end{lem}
\begin{proof}
    By definition, $\psi_{\alpha,\beta}$ is given by right multiplication by the matrix $\sum_{i,j=1}^n F_{i,j}\otimes k^{\alpha,\beta}(z_j,z_i)$. Hence, by \cite[Lemma 4.31]{horn1994topics}, $[\psi_{\alpha,\beta}]_{\varepsilon}=\sum_{i,j=1}^n F_{j,i}\otimes k^{\alpha,\beta}(z_j,z_i)^t\otimes I_r$.
    \\\\
    For the second part, by definition, $\varphi_W$ acts on the $i$'th block by multiplication on the left by $W_i$. Therefore, again by \cite[Lemma 4.31]{horn1994topics}, $[\varphi_W]_{\varepsilon}$ is a block diagonal matrix, with the $i$'th block equal to $I_m\otimes W_i$. Therefore,
    $$[\varphi_W]_{\varepsilon}=\sum_{i=1}^n F_{i,i}\otimes I_m\otimes W_i.$$
\end{proof}
\begin{remark}
    Immediately, $\varepsilon$ is an orthonormal basis with respect to the inner product defined on $M_{r\times mn}(\C)$.\label{remark:orthonormal.basis}
\end{remark}
\begin{lem}
    For every $(\alpha,\beta)\in\G(m\times m)$, $\pi_{\alpha,\beta}^{(n)}(\Psi)\otimes I_r=[\psi_{\alpha,\beta}]_{\varepsilon}$. In particular, $[\psi_{\alpha,\beta}]_{\varepsilon}$ is positive definite.
    \label{lem.ampliation.of.pi}
\end{lem}
\begin{proof}
    $$(\pi_{\alpha,\beta}^{(n)}(\Psi))_{i,j}=\pi_{\alpha,\beta}(\ell_{i,j})=(\alpha^t+f_{\lambda}(z_i)\beta^t)(\overline{\alpha}+\overline{f_{\lambda}(z_j)\beta})+\frac{B_{\lambda}(z_i)\overline{B_{\lambda}(z_j)}}{1-z_i\overline{z_j}}I_m=$$$$=k^{\alpha,\beta}(z_i,z_j)^t.$$ Therefore,
    $$\pi_{\alpha,\beta}^{(n)}(\Psi)\otimes I_r=\sum_{i,j=1}^n F_{j,i}\otimes k^{\alpha,\beta}(z_j,z_i)^t\otimes I_r=[\psi_{\alpha,\beta}]_{\varepsilon}.$$
    For the second part, by \ref{prop:positivity.of.psi} and since $\pi^{(n)}$ is a homomorphism of unital $C^*-$ algebras, $\pi^{(n)}_{\alpha, \beta}(\Psi)$ is positive definite. By properties of Kronecker product and since $I_r$ is positive definite, $[\psi_{\alpha,\beta}]_{\varepsilon}$ is positive definite. Lastly, by remark \ref{remark:orthonormal.basis}, $\psi_{\alpha,\beta}$ is positive definite. 
\end{proof}
\begin{lem}
    The generalized trace condition for fixed $m,r\in\N$, $(\alpha,\beta)\in\G(m\times m)$ and  $W_1,\ldots, W_n\in M_r(\C)$ is equivalent to 
    $$\left\|\left[\psi_{\alpha,\beta}^{-\frac{1}{2}}\right]_\varepsilon\cdot\left[\varphi_W\right]_\varepsilon\cdot \left[\psi_{\alpha,\beta}^{\frac{1}{2}}\right]_\varepsilon\right\|\leq 1.$$ 
\end{lem}
\begin{proof}
    By our previous notation, the generalized trace condition is
    \begin{equation}
    \begin{aligned}
        &\operatorname{tr}\left[
        \begin{pmatrix}
            X_1 & \dots & X_n
        \end{pmatrix}
        \left[k^{\alpha,\beta}(z_j, z_i)\right]_{i,j=1}^n
        \begin{pmatrix}
            X_1^* \\ \vdots \\ X_n^*
        \end{pmatrix}
        \right. \\
        &\quad -
        \left.
        \begin{pmatrix}
            W_1^*X_1 & \dots & W_n^*X_n
        \end{pmatrix}
        \left[k^{\alpha,\beta}(z_j, z_i)\right]_{i,j=1}^n
        \begin{pmatrix}
            X_1^*W_1 \\ \vdots \\ X_n^*W_n
        \end{pmatrix}
        \right] \geq 0,
    \end{aligned}
\end{equation}
for all $X_1,\ldots,X_n\in M_{r\times m}(\C)$.
Denoting $X=
\begin{pmatrix}
X_1&\ldots&X_n\end{pmatrix}$, this is the same as
$$\ip{\psi_{\alpha,\beta}(X)-\varphi_W(\psi_{\alpha,\beta}(\varphi_W^*(X)))}{X}\geq 0,$$ for all $X_1,\ldots, X_n\in M_{r\times m}(\C)$. By definition, this is the same as
$$\psi_{\alpha,\beta}\geq \varphi_W\circ\psi_{\alpha,\beta}\circ\varphi_W^*.$$
Using Lemma \ref{lem.ampliation.of.pi}, we can conjugate both sides by $\psi_{\alpha,\beta}^{-\frac{1}{2}}$. Consequently, the trace condition is equivalent to
$$1\geq \left(\psi_{\alpha,\beta}^{-\frac{1}{2}}\circ\varphi_W\circ\psi_{\alpha,\beta}^{\frac{1}{2}}\right)\left(\psi_{\alpha,\beta}^{-\frac{1}{2}}\circ\varphi_W\circ\psi_{\alpha,\beta}^{\frac{1}{2}}\right)^*.$$
The latter, in turn, is equivalent to
$$\left\|\psi_{\alpha,\beta}^{-\frac{1}{2}}\circ\varphi_W\circ\psi_{\alpha,\beta}^{\frac{1}{2}}\right\|\leq 1.$$
\end{proof}

We are ready to prove Theorem \ref{thm:completely.isometric.embedding}.

\begin{proof}[Proof of Theorem \ref{thm:completely.isometric.embedding}]
    First, it is clear that the range is contained in $M_n(\frak{B})$, as $\Psi\in\frak{B}$. Throughout this proof, we will use the notation
    $$\widetilde{\psi_{\alpha,\beta}}:=\sum_{i,j=1}^k F_{j,i}\otimes k^{\alpha,\beta}(z_j,z_i)^t.$$
    
    Let $r\in\N$ and consider the homomorphism $\Gamma^{(r)} \colon M_r(H^{\infty}_{\text{node}}/I)\to M_r(M_n(\frak{A}))$
    given by
    $$\Gamma^{(r)}([f_{i,j}+I])=\sum_{i,j=1}^r E_{i,j}\otimes \Gamma(f_{i,j}+I).$$
    To show $\Gamma^{(r)}$ is an isometry, it is enough to prove
    $\|[f_{i,j}+I]\|\leq 1\iff\|\Gamma^{(r)}([f_{i,j}+I])\|\leq 1$.
    As $\frak{A}$ is RFD, our claim is equivalent to showing that
    $\|[f_{i,j}+I]\|\leq 1\iff\|\tilde{\rho}(\Gamma^{(r)}([f_{i,j}+I]))\|\leq 1$
    for every finite-dimensional representation
    $\tilde{\rho}\colon M_r(M_n(\frak{A}))\to M_k(\C)$.
    Fix a finite-dimensional representation $\tilde{\rho}$ of $M_r(M_n(\frak{A}))$. It is well known that 
    $\tilde{\rho}$ is unitarily equivalent to $\rho^{(r)}$ for some representation 
    $$\rho \colon M_n(\frak{A})\to M_{s}(\C).$$
    Hence, $\left\|\rho^{(r)}(\Gamma^{(r)}([f_{i,j}+I]))\right\|=\left\|\tilde{\rho}(\Gamma^{(r)}([f_{i,j}+I]))\right\|$.
    Again, $\rho$ is unitarily equivalent to $\hat{\rho}^{(n)}$ where $\hat{\rho}$ is a representation of $\frak{A}$. As we noted, $\hat{\rho}=\pi_{\alpha,\beta}$ for some $(\alpha,\beta)\in\G(m\times m)$. Hence, $\rho$ is unitarily equivalent to $\pi_{\alpha,\beta}^{(n)}$. Let $U$ be the unitary satisfying $\rho([f_{p,q}])=U\pi^{(n)}_{\alpha,\beta}([f_{p,q}])U^*$.
    Then,
    $$\rho^{(r)}\left(\Gamma^{(r)}([f_{i,j}+I])\right)=\sum_{i,j=1}^r E_{i,j}\otimes\rho\left(\Gamma\left(f_{i,j}+I\right)\right)=\sum_{i,j=1}^r E_{i,j}\otimes\left(U\pi_{\alpha,\beta}^{(n)}\left(\Gamma\left(f_{i,j}+I\right)\right)U^*\right)=$$
    $$=\sum_{i,j=1}^r E_{i,j}\otimes \left(U\pi_{\alpha,\beta}^{(n)}\left(\Psi^{-\frac{1}{2}}D_{f_{i,j}}\Psi^{\frac{1}{2}}\right)U^*\right)=
    $$$$=\sum_{i,j=1}^r E_{i,j}\otimes\left(U\widetilde{\psi_{\alpha,\beta}}^{-\frac{1}{2}}\left(\sum_{k=1}^n F_{k,k}\otimes f_{i,j}(z_k)I_m\right)\widetilde{\psi_{\alpha,\beta}}^{\frac{1}{2}}U^*\right)=$$
    $$=\left(I_r\otimes U^*\right)\left(I_r\otimes \widetilde{\psi_{\alpha,\beta}}^{-\frac{1}{2}}\right)\left(\sum_{i,j=1}^r E_{i,j}\otimes \left(\sum_{k=1}^n F_{k,k}\otimes f_{i,j}(z_k)I_m\right)\right)\left(I_r\otimes\widetilde{\psi_{\alpha,\beta}}^{\frac{1}{2}}\right)\left(I_r\otimes U\right).$$
    Note that the last equality of the second line holds since $\pi_{\alpha,\beta}^{(n)}(f_{i,j}(z_k))=f_{i,j}(z_k)I_m$. Since $I_r$ and $U$ are unitaries, $I_r\otimes U$ is unitary (with $(I_r\otimes U)^*=I_r\otimes U^*$) by properties of Kronecher product, that is  $\rho^{(r)}\left(\Gamma^{(r)}([f_{i,j}])\right)$ is unitarily equivalent to
    \begin{equation}
        \label{eq:unitary.equivalence}
        \left(I_r\otimes \widetilde{\psi_{\alpha,\beta}}^{-\frac{1}{2}}\right)\left(\sum_{i,j=1}^r E_{i,j}\otimes \left(\sum_{k=1}^n F_{k,k}\otimes f_{i,j}(z_k)I_m\right)\right)\left(I_r\otimes\widetilde{\psi_{\alpha,\beta}}^{\frac{1}{2}}\right).
    \end{equation}
    Conjugating \eqref{eq:unitary.equivalence} by the unitary that interchanges the first two tensor products and then by the unitary that interchanges the last two tensor products, we get 
    $$\left(\widetilde{\psi_{\alpha,\beta}}^{-\frac{1}{2}}\otimes I_r\right)\left(\sum_{k=1}^n F_{k,k}\otimes I_m\otimes W_k\right)\left(\widetilde{\psi_{\alpha,\beta}}^{\frac{1}{2}}\otimes I_r\right)=\left[\psi_{\alpha,\beta}\right]^{-\frac{1}{2}}_{\varepsilon}\cdot\left[\varphi_W\right]_{\varepsilon}\cdot \left[\psi_{\alpha,\beta}\right]^{-\frac{1}{2}}_{\varepsilon}=$$
    $$=\left[\psi_{\alpha,\beta}^{-\frac{1}{2}}\circ\varphi_W\circ\psi_{\alpha,\beta}^{\frac{1}{2}}\right]_{\varepsilon},$$
 
    with $W_k=[f_{i,j}](z_k)$.
    Now assume that $\left\|\tilde{\rho}(\Gamma^{(r)}(f))\right\|\leq 1$. By unitary equivalence, and since $\varepsilon$ is an orthonormal basis,
    $$\left\|\left[\psi_{\alpha,\beta}^{-\frac{1}{2}}\circ\varphi_W\circ\psi_{\alpha,\beta}^{\frac{1}{2}}\right]_{\varepsilon}\right\|=\left\|\psi_{\alpha,\beta}^{-\frac{1}{2}}\circ\varphi_W\circ\psi_{\alpha,\beta}^{\frac{1}{2}}\right\|\leq 1.$$
    We have observed that this is equivalent to the generalized trace condition \eqref{eq:gen.trace.con} for $W_1,\ldots, W_n$ being true for $r$ and $m$.
    As $M_r\left(H^{\infty}_{\text{node}}/I\right)$ and $M_r(H^{\infty}_{\text{node}})/\mathcal{I}$ are isometrically isomorphic, by \eqref{eq:gen.trace.con} this is equivalent to 
    $\left\|[f_{i,j}]+\mathcal{I}\right\|=\left\|[f_{i,j}+I]\right\|\leq 1$.
    Therefore, $\tilde{\rho}(\Gamma^{(r)})$ is an isometry, hence $\Gamma^{(r)}$ is an isometry. 
\end{proof}

\begin{quest}
Is $C^*(\Gamma(\node/I))$ the $C^*$-envelope of $\node/I$ at least for $n$ big enough? 
\end{quest}

\begin{remark}
It follows from a general proposition in \cite{davidson2009constrained} that if we have only two points, then $C^*_e(\node/I) \cong M_2(\C)$. Hence, we only ask the above question for $n \geq 3$. However, maybe, a bigger $n$ is necessary to ensure that $C^*_e(\node/I) \cong C^*(\Gamma(\node/I))$. It seems that a big obstruction in understanding the above question is a technical one. Namely, figuring out whether $\Psi \in C^*(\Gamma(\node/I))$. If it is true, then one might be able to use \cite[Theorem]{Hopenwasser-boundary}.
\end{remark}

\bibliographystyle{abbrv}
\bibliography{mybib}
\end{document}